\documentclass[preprint]{imsart}
\arxiv{1510.08029}

\usepackage[numbers]{natbib}
\usepackage{hyperref}
\usepackage{cleveref}
\usepackage{enumitem}

\usepackage[T1]{fontenc}
\usepackage{lmodern}
\usepackage{graphicx}
\usepackage{algorithm2e}
\usepackage{amssymb,amsmath,amsthm}
\usepackage{subfigure}
\usepackage{ifxetex,ifluatex}
\usepackage{fixltx2e} 
\IfFileExists{microtype.sty}{\usepackage{microtype}}{}
\IfFileExists{upquote.sty}{\usepackage{upquote}}{}
\ifnum 0\ifxetex 1\fi\ifluatex 1\fi=0 
  \usepackage[utf8]{inputenc}

\usepackage[a4paper, total={5.2in, 8.2in}]{geometry}

\usepackage{thmtools}
\usepackage{mathtools}

\usepackage{algorithmicx}
\hypersetup{breaklinks=true,
    linkcolor=red,
    citecolor=blue,
    colorlinks=true,
pdfborder={0 0 0}}

\usepackage{amsthm}
\usepackage{thmtools}
\usepackage{mathtools}

\usepackage{autonum}

\numberwithin{equation}{section}
\startlocaldefs

\crefname{thm}{theorem}{theorems}
\crefname{lemma}{lemma}{lemmas}
\crefname{prop}{proposition}{propositions}
\crefname{assumption}{assumption}{assumptions}
\crefname{example}{example}{examples}
\crefname{cor}{corollary}{corollaries}

\declaretheorem[name=Theorem,numberwithin=section]{thm}
\declaretheorem[name=Proposition,sibling=thm]{prop}
\declaretheorem[name=Lemma,sibling=thm]{lemma}
\declaretheorem[name=Corollary,sibling=thm]{cor}

\declaretheorem[name=Example,style=definition]{example}

\DeclareMathOperator*{\argmin}{argmin}

\DeclareMathOperator*{\supp}{supp}

\newcommand{\design}{\mathbb{X}}

\newcommand{\E}{\mathbb{E}} 
\newcommand{\Evmu}{\mathbb{E}_\vmu} 
\newcommand{\R}{\mathbb{R}}
\newcommand{\Rn}{\mathbb{R}^n}
\newcommand{\Rm}{\mathbb{R}^m}

\newcommand{\proba}[1]{\mathbb{P}\left( #1 \right)}

\newcommand{\conebetaat}[1]{{\mathcal{S}_n^{[ #1 ]}}}
\newcommand{\conebeta}{{\conebetaat{\beta}}}
\newcommand{\conebetak}[1]{{\mathcal{S}_{ #1 }^{[ \beta ]}}}

\newcommand{\vmu}{{\boldsymbol{\mu}}}
\newcommand{\vtheta}{{\boldsymbol{\theta}}}
\newcommand{\vv}{{\boldsymbol{v}}}
\newcommand{\ve}{{\boldsymbol{e}}}
\newcommand{\vzero}{{\boldsymbol{0}}}
\newcommand{\vx}{{\boldsymbol{x}}}
\newcommand{\htheta}{{\boldsymbol{\hat \theta}}}
\newcommand{\vu}{{\boldsymbol{u}}}
\newcommand{\hmu}{\boldsymbol{\hat \mu}}

\newcommand{\vpi}{{\boldsymbol{\pi}}}
\newcommand{\tildemu}{{\boldsymbol{\tilde \mu}}}

\newcommand{\vy}{\mathbf{y}}
\newcommand{\vxi}{{\boldsymbol{\xi}}}
\newcommand{\vg}{{\boldsymbol{g}}}

\DeclareMathOperator{\Ima}{Im}

\newcommand{\increasings}{{\mathcal{S}_n^\uparrow}}
\newcommand{\increasingsat}{{\mathcal{S}^\uparrow}}

\newcommand{\convexs}{\mathcal{S}_n^{\textsc{c}}}

\newcommand{\scalednorm}[1]{\Vert  #1 \Vert}
\newcommand{\scalednorms}[1]{\scalednorm{ #1 }^2}

\newcommand{\inftynorm}[1]{\vert #1 \vert_\infty}

\newcommand{\hmuora}{
    \hmu^{\textsc{oracle}}
}

\newcommand{\vomega}{{\boldsymbol{\omega}}}
\newcommand{\euclidnorm}[1]{\vert  #1 \vert_2}
\newcommand{\euclidnorms}[1]{\vert  #1 \vert_2^2}
\newcommand{\K}{{\mathcal{K}}}
\newcommand{\C}{{\mathcal{C}}}
\newcommand{\ls}{\hmu^{\textsc{ls}}}

\newcommand{\lineality}[1]{{\mathrm{Lin}}( #1 )}

\newcommand{\coneleft}{C^L}
\newcommand{\coneright}{C^R}

\endlocaldefs

\begin{document}

\begin{abstract}
    The performance of Least Squares (LS) estimators is studied
    in isotonic, unimodal and convex regression.
    Our results have the form of sharp oracle inequalities that account
    for the model misspecification error.
    In isotonic and unimodal regression, 
    the LS estimator achieves the nonparametric rate $n^{-2/3}$
    as well as a parametric rate of order $k/n$ up to logarithmic factors,
    where $k$ is the number of constant pieces of the true parameter.

    In univariate convex regression, the LS estimator satisfies an adaptive risk bound of order $q/n$ up to logarithmic factors,
    where $q$ is the number of affine pieces of the true regression function.
    This adaptive risk bound holds for any design points.
    While \citet{guntuboyina2013global} established that the nonparametric rate
    of convex regression is of order $n^{-4/5}$ for equispaced design points,
    we show that the nonparametric rate of convex regression can be as slow
    as $n^{-2/3}$ for some worst-case design points.
    This phenomenon can be explained as follows:
    Although convexity brings more structure than unimodality,
    for some worst-case design points this extra structure 
    is uninformative and the nonparametric rates of unimodal regression and convex regression are both $n^{-2/3}$.
\end{abstract}

\title{Sharp oracle inequalities for Least Squares estimators in shape restricted regression}

\begin{aug}
    \author{\fnms{Pierre C.} \snm{Bellec}
        \thanksref{ecodec}
        \ead[label=e1,mark]{pierre.bellec@ensae.fr}
    }
    \affiliation{ENSAE and UMR CNRS 9194}
    \address{ENSAE,\\ 3 avenue Pierre Larousse, \\ 92245 Malakoff Cedex, France. }
    \thankstext{ecodec}{
        This work was
        supported by GENES and by the grant Investissements d'Avenir
    (ANR-11-IDEX-0003/Labex Ecodec/ANR-11-LABX-0047).}
\end{aug}

\date{\today}

\maketitle

\section{Introduction}
Assume that we have the observations
\begin{equation}
    \label{eq:observations-restricted}
    Y_i = \mu_i + \xi_i, \qquad i=1,...,n,
\end{equation}
where $\vmu =(\mu_1,...,\mu_n)^T \in\Rn$ is unknown,
$\vxi = (\xi_1,...,\xi_n)^T$ is a noise vector with $n$-dimensional Gaussian distribution
$\mathcal{N}(\vzero,\sigma^2 I_{n\times n})$ where $\sigma>0$
and $I_{n\times n}$ is the $n\times n$ identity matrix.
We will also use the notation $\vg \coloneqq (1/\sigma)\vxi$ so that
$\vy = \vmu + \vxi = \vmu + \sigma\vg$ and $\vg\sim\mathcal N(\vzero,I_{n\times n})$.
Denote by $\Evmu$ and $\mathbb P_\vmu$ the expectation and the probability
with respect to the distribution of the random variable $\vy=\vmu+\vxi$.
The vector $\vy = (Y_1,...,Y_n)^T$ is observed and the goal is to estimate $\vmu$.
The estimation error is measured
with the scaled norm $\scalednorm{\cdot}$ defined by 
 \begin{equation}
     \scalednorms{\vu} = \frac{1}{n} \sum_{i=1}^n u_i^2, \qquad \vu = (u_1,...,u_n)^T\in\Rn.
\end{equation}
The error  of an estimator $\hmu$ of $\vmu$ 
is given by $\scalednorms{\hmu - \vmu}$. 
Let also $\inftynorm{\cdot}$ be the infinity norm
and $\euclidnorm{\cdot}$ be the Euclidean norm, so that $\frac{1}{n} \euclidnorms{\cdot} = \scalednorms{\cdot}$.

This paper studies the Least Squares (LS) estimator
in shape restricted regression under model misspecification.
The LS estimator
over a nonempty closed set $K\subset\R^n$ is defined by
\begin{equation}
    \ls(K) \in \argmin_{\vu\in K} \scalednorms{\vy - \vu}.
\end{equation}
Model misspecification allows that the true parameter $\vmu$
does not belong to $K$.
There is a large literature on the performance of the LS estimator
in isotonic and convex regression, that is,
when the set $K$ is the set of all nondecreasing sequences
or the set of convex sequences.
Some of these results are reviewed in the following subsections.

\subsection{Isotonic regression}
\label{intro:increasing}

Let $\increasings$ be the set of all nondecreasing sequences,
defined by
\begin{equation}
    \increasings \coloneqq \{\vu=(u_1,...,u_n)^T\in\Rn:  u_i \le u_{i+1}, \quad i=1,...,n-1\}.
\end{equation}
The set $\increasings$ is a closed convex cone.
Two quantities are useful to describe the performance of the LS estimator $\ls(\increasings)$.
First, define the total variation by
\begin{equation}
    V(\vtheta) \coloneqq \max_{i=1,...,n}\theta_i - \min_{i=1,...,n}\theta_i,
    \qquad
    \vtheta=(\theta_1,...,\theta_n)^T\in\R^n.
    \label{eq:def-V}
\end{equation}
If $\vu=(u_1,\dots,u_n)^T \in \increasings$,
its total variation is simply $V(\vu) = u_n - u_1$.
Second, for $\vu=(u_1,\dots,u_n)^T \in \increasings$,
let $k(\vu)\ge 1$ be the integer such that $k(\vu)-1$ is the number
of inequalities $u_i\le u_{i+1}$ that are strict for $i=1,\dots,n-1$ (the number of jumps of $\vu$).

Previous results on the performance of the LS estimator
$\ls(\increasings)$
can be found 
in \cite{meyer2000degrees,zhang2002risk,chatterjee2013risk,chatterjee2014new}, where risk bounds or oracle inequalities with leading 
constant strictly greater than 1 are derived.
Two types of risk bounds or
oracle inequalities have been obtained so far.
If, $\vmu=(\mu_1,...,\mu_n)^T\in\increasings$, it is known
\cite{meyer2000degrees,zhang2002risk,chatterjee2013risk,chatterjee2014new} that
for some absolute constant $c > 0$,
\begin{equation}
\label{eq:rate23-example-nonoi}
    \Evmu \scalednorms{\ls(\increasings) - \vmu}
    \le \frac{c \sigma^2 \log(en)}{n}
    + c \sigma^2 \left(\frac{V(\vmu)}{\sigma n}\right)^{2/3}
\end{equation}
and $c\le 12.3$, cf. \cite{zhang2002risk}.
If $\vmu\in\increasings$,
the following oracle inequality was proved in \cite{chatterjee2013risk}:
\begin{equation}
\label{eq:increasings-adapt-example-oi}
    \Evmu \scalednorms{\ls(\increasings) - \vmu}
    \le 6 
        \min_{\vu\in\increasings}
        \left(
            \scalednorms{\vu - \vmu}
            +
            \frac{\sigma^2 k(\vu)}{n}\log\frac{en}{k(\vu)}
        \right).
\end{equation}
The risk bounds \eqref{eq:rate23-example-nonoi} and
\eqref{eq:increasings-adapt-example-oi} hold under the assumption
that $\vmu\in\increasings$,
which does not allow for any model misspecification.
We will see below that this assumption can be dropped.
The oracle inequality
\eqref{eq:rate23-example-nonoi} implies that the LS estimator
achieves the rate $n^{-2/3}$ while
\eqref{eq:increasings-adapt-example-oi} yields
a parametric rate (up to logarithmic factors) if $\vmu$ is well approximated
by a piecewise constant sequence with not too many pieces.
Let us note that the bound \eqref{eq:increasings-adapt-example-oi} can be used
to obtain that $\ls(\increasings)$ converges at the rate $n^{-2/3}$ up to logarithmic factors,
thanks to the approximation argument given in \cite[Lemma 2]{bellec2015sharp}.

Mimimax lower bounds that match \eqref{eq:rate23-example-nonoi}
and \eqref{eq:increasings-adapt-example-oi} up
to logarithmic factors have been obtained 
in \cite{chatterjee2013risk,bellec2015sharp}.
If $D>0$ is a fixed parameter
and
$\log(en)^3 \sigma^2 \le  n D^2$,
the bound
\eqref{eq:rate23-example-nonoi}
yields the rate $(D\sigma^2)^{2/3}n^{-2/3}$
for the risk of $\ls(\increasings)$.
By the lower bound \cite[Corollary 5]{bellec2015sharp},
this rate is minimax optimal
over the class $\{\vmu\in\increasings: V(\vu) \le D \}$
if
$\log(en)^3 \sigma^2 \le  n D^2$.
Proposition 4 in \cite{bellec2015sharp} shows
that there exist absolute constants $c,c'>0$ such that
for any estimator $\hmu$,
\begin{equation}
    \sup_{\vmu\in\increasings:k(\vmu)\le k}
    \mathbb P_{\vmu}
    (
        \scalednorms{\hmu - \vmu}
        \ge c\sigma^2 k /n
    ) \ge c'.
    \label{eq:minimax-lower-increasings}
\end{equation}
Together, \eqref{eq:increasings-adapt-example-oi} and \eqref{eq:minimax-lower-increasings}
establish that for any $k=1,...,n$,
the minimax rate over the class $\{\vmu\in\increasings: k(\vmu)\le k \}$ is of
order $\sigma^2 k / n$ up to logarithmic factors.

\subsection{Convex regression with equispaced design points}
\label{intro:convex}

If $n\ge 3$, define the set of convex sequences $\convexs$ by
\begin{align}
    \convexs 
    & \coloneqq \{ \vu=(u_1,\dots,u_n)^T\in\Rn: \; 2u_i \le u_{i+1} + u_{i-1}, \  i=2,\dots,n-1 \}.
\end{align}
For $\vu=(u_1,\dots,u_n)^T \in \convexs$,
let $q(\vu)\in\{1,...,n-1\}$ be the smallest integer $q$
such that there exists a partition $(T_1,...,T_q)$ of $\{1,...,n\}$
and real numbers $a_1,...,a_q$ satisfying
\begin{equation}
    u_i = a_j(i-l) + u_{l},
    \qquad
    i,l\in T_j,
    \qquad
    j=1,...,q.
\end{equation}
The quantity $q(\vu)$ is the smallest integer $q$ such that $\vu$ 
is piecewise affine with $q$ pieces.
If $x_1 < ... < x_n$ are equispaced design points in $\R$, i.e.,
$x_i = (i-1) ( x_2 - x_1 ) + x_ 1$, $i=2,...,n$,
then
\begin{equation}
    \convexs = \{
        \vu\in\Rn,
        \vu = 
        (f(x_1),
            ...,
        f(x_n))^T
        \text{ for some convex function }
        f:\R\rightarrow\R 
    \}.
\end{equation}
The performance of the LS estimator over convex sequences
has been recently studied in \cite{guntuboyina2013global,chatterjee2013risk}, where it was proved that
if $\vmu\in\convexs$,
the estimator $\hmu=\ls(\convexs)$ satisfies
\begin{align}
    \Evmu \scalednorms{\hmu - \vmu}
    & \le 
    C \left(
        \min_{\vu\in\convexs}
        \left(
            \scalednorms{\vu - \vmu}
            +
            \frac{\sigma^2 q(\vu)}{n}\left(\log\frac{en}{q(\vu)}\right)^{5/4}
        \right)
    \right)
    \label{eq:convexs-adapt-example-oi}
\end{align}
for some absolute constant $C>0$.
If $\vmu\in\convexs$ and $nR_\vmu^2 \ge \log(en)^{5/4} \sigma^2$
where $R_\vmu$ 
defined in \Cref{cor:rate45} below
is a constant that depends only on $\vmu$,
then the estimator $\hmu=\ls(\convexs)$ satisfies
\begin{align}
    \Evmu \scalednorms{\hmu - \vmu}
    & \le C
     \left(\frac{\sqrt{R_{\vmu}} \sigma^2}{n}\right)^{4/5}  \log(en) 
    \label{eq:rate45-example-nonoi}
\end{align}
for some absolute constant $C>0$.
The bound 
\eqref{eq:convexs-adapt-example-oi}
yields an almost parametric rate
if $\vmu$ can be well approximated by a piecewise affine
sequence with not too many pieces.
If $\bar R>0$ is a fixed parameter
and
$n\bar R^2 \ge \log(en)^{5/4} \sigma^2$,
the bound
\eqref{eq:rate45-example-nonoi}
yields the rate $(\bar R^2\sigma^8)^{1/5}n^{-4/5} \log(en)$, which is minimax optimal
over the class $\{\vmu\in\convexs: R_\vmu \le \bar R \}$
up to logarithmic factors \cite{guntuboyina2013global}.

The above results hold in convex regression for equispaced design
points. The following subsection introduces the notation
that will be used to study convex regression with
non-equispaced design points.

\subsection{Non-equispaced design points in convex regression}
If $x_1 < ... < x_n$ are non-equispaced design points in $\R$,
define the cone
\begin{align}
    \K^C_{x_1,...,x_n}
    &\coloneqq \{
        \vu\in\Rn,
        \vu = 
        (f(x_1),
            ...,
        f(x_n))^T
        \text{ for some convex function }
        f:\R\rightarrow\R 
    \}.
\end{align}
This can be rewritten as
\begin{align}
    \K^C_{x_1,...,x_n}
    &\coloneqq \left\{
        \vu\in\Rn:
        \tfrac{u_i - u_{i-1}}{x_i - x_{i-1}}
        \le
        \tfrac{u_{i+1} - u_{i}}{x_{i+1} - x_{i}},
        i=2,...,n-1
    \right\}.
    \label{eq:def-convex-x_i}
\end{align}
For any $\vu = (u_1,...,u_n)^T \in\K^C_{x_1,...,x_n}$,
we say that $\vu$ is piecewise affine with
$k$ pieces if there exist
real numbers $a_1,...,a_k$
and a partition
$(T_1,...,T_k)$ of $\{1,..,n\}$
such that 
\begin{equation}
    u_i = a_j(x_i - x_{l}) + u_{l},
    \qquad
    i,l\in T_j,
    \qquad
    j=1,...,k.
    \label{eq:piecewise-affine}
\end{equation}
If $\vu=(f(x_1),...,f(x_n))^T$
for some convex function $f:\R\rightarrow\R$ and $f$ is a piecewise affine function with $k$ pieces,
then $\vu$ is piecewise affine with $k$ pieces.
For any $\vu\in\K^C_{x_1,...,x_n}$, let $q(\vu)\ge 1$ be the
smallest integer such that $\vu$
is piecewise affine with $q(\vu)$ pieces.
The quantity $q(\vu)\ge 1$ satisfies
\begin{equation}
    q(\vu) - 1 \le
    \left|\left\{
        i=2,...,n-1:\quad
        \frac{u_i - u_{i-1}}{x_i - x_{i-1}}
        < 
        \frac{u_{i+1} - u_{i}}{x_{i+1} - x_{i}}
    \right\} \right|.
\end{equation}
The performance of the LS estimator $\ls(\K^C_{x_1,...,x_n})$ is also studied in \cite{guntuboyina2013global}
in the case where the design points are almost equispaced:
The bounds 
\eqref{eq:convexs-adapt-example-oi}
and
\eqref{eq:rate45-example-nonoi}
both hold if $\convexs$ is replaced with $\K^C_{x_1,...,x_n}$
and if $C>0$ is a constant that depends on the ratio
\begin{equation}
    \label{eq:ratio}
    \frac{\max_{i=2,...,n} (x_i-x_{i-1})}{
        \min_{i=2,...,n} (x_i-x_{i-1})
    },
\end{equation}
and this constant $C$ becomes arbitrarily large
as this ratio tends to infinity.

Although \eqref{eq:rate45-example-nonoi}
and
\eqref{eq:convexs-adapt-example-oi}
provide an accurate picture of the performance 
of the LS estimator for equispaced (or almost equispaced)
design points, it is not known whether
these bounds continue to hold for other design points.
A goal of the present paper is to fill this gap.
\Cref{s:convex} shows that the oracle inequality
\eqref{eq:convexs-adapt-example-oi} holds 
irrespective of the design points,
while the nonparametric rate of the LS estimator can be as slow as
$n^{-2/3}$ for some worst-case design points.

It is clear that a convex function is unimodal in the sense that it is 
first non-increasing and then nondecreasing.
The following subsection introduces the set of unimodal sequences,
and \Cref{ss:convex-unimodal} studies the relationship
between convex regression and unimodal regression.

\subsection{Unimodal regression}
Let $m=1,...,n$.
A sequence $\vu\in\R^n$
is unimodal with mode at position $m$
if and only if $\vu_{\{1,...,m\}}$ is non-increasing
and $\vu_{\{m,...,n\}}$ is nondecreasing.
Define the convex set
\begin{equation}
    K_m \coloneqq \{ \vu=(u_1,...,u_n)^T\in\R^n: u_1 \ge ... \ge u_m \le u_{m-1} \le ... \le u_n \}.
    \label{eq:def-Km-unimodal}
\end{equation}
The convex set $K_m$ is the set of all unimodal sequences with mode at position $m$
and
\begin{equation}
    \mathcal U \coloneqq \cup_{m=1,...,n} K_m
\end{equation}
is the set of all unimodal sequences.
The set $\mathcal U$ is non-convex.
For all $\vu\in\mathcal U$, let $k(\vu)$ be the smallest integer $k$ such that
$\vu$ is piecewise constant with $k$ pieces, i.e.,
the smallest integer $k$ such that there exists a partition $(T_1,...,T_k)$ of $\{1,...,n\}$
such that for all $l=1,...,k$,
\begin{itemize}
    \item the sequence $\vu_{T_l}$ is constant, and
    \item the set $T_l$ is convex in the sense that if $a,b\in T_l$
        then $T_l$ contains all integers between $a$ and $b$.
\end{itemize}
If $\vu\in\increasings$, this definition of $k(\vu)$
coincides with that defined above.

As the inclusion $\increasings\subset\mathcal U$ holds, 
the lower bound \eqref{eq:minimax-lower-increasings} implies that
for any estimator $\hmu$,
\begin{equation}
    \sup_{\vmu\in\mathcal U :k(\vmu)\le k}
    \mathbb P_{\vmu}
    (
        \scalednorms{\hmu - \vmu}
        \ge c\sigma^2 k /n
    ) \ge c' > 0.
    \label{eq:minimax-lower-U}
\end{equation}
\citet{chatterjee2015adaptive}
recently obtained an adaptive risk bound
of the form
\begin{equation}
    \mathbb P\left(
        \scalednorms{\ls(\mathcal U) - \vmu}
        \le
        \frac{C\sigma^2}{n}\Big( k(\vu) \log(en) \Big)^{3/2}
    \right) \ge 1 - \frac 1 n,
    \label{eq:unimodal-32}
\end{equation}
where $C>0$ is an absolute constant.
This risk bound does not match the lower bound \eqref{eq:minimax-lower-U}
because of the exponent $3/2$.

\subsection{Organisation of the paper}

\Cref{s:preliminary} recalls properties of closed convex set and closed convex cones.
\begin{itemize}
    \item 
        \emph{General oracle inequalities.}
        In \Cref{s:tools-soi} we establish general tools
        that yield sharp oracle inequalities:
        \Cref{cor:tangent-cone} and \Cref{thm:isomorphic}.
    \item
        \emph{Sharp bounds in isotonic regression.}
        In \Cref{s:isotonic} we apply results of \Cref{s:tools-soi}
        to the isotonic LS estimator. We obtain an adaptive risk bound
        that is tight with sharp numerical constants.
    \item
        \emph{On the relationship between unimodal and convex regression.}
        \Cref{s:convex} studies the role of the design points in univariate 
        convex regression:
        Although the nonparametric rate is of order $n^{-4/5}$ for equispaced design points,
        this rate can be as slow as $n^{-2/3}$ for some worst-case design points that are studied in \Cref{s:convex},
        whereas the adaptive risk bound \eqref{eq:convexs-adapt-example-oi} holds for any design points.
        The relation between convex regression and unimodal regression is discussed in \Cref{ss:convex-unimodal}:
        Although convexity brings more structure than unimodality,
        for some worst-case design points this extra structure 
        is uninformative and the nonparametric rates of unimodal regression and convex regression are both $n^{-2/3}$.
        \Cref{s:unimodal} studies unimodal regression
        and improves some of the results of \cite{chatterjee2015adaptive}
        on the performance of the unimodal LS estimator.
    \item
        \emph{Comparison of different misspecification errors.}
        In \Cref{s:misspecification} we compare
        different quantities that represent the estimation error
        when the model is misspecified.
        In particular, \Cref{s:misspecification} explains that if $K$ is a closed convex set and $\vmu\notin K$,
        the sharp oracle inequalities obtained in \Cref{s:tools-soi,s:isotonic,s:convex}
        yield upper bounds on the estimation error $\scalednorm{\ls(K) - \Pi_K(\vmu)}$.
        If $\vmu\notin K$, the LS estimator consistently estimates the
        projection of the true parameter $\vmu$ onto $K$
        for $K=\increasings$ and $K=\convexs$.
\end{itemize}
Some proofs are delayed to \Cref{s:proofs-unimodal-to-convex,s:proof-iso}.

\subsection{Preliminary properties of closed convex sets}
\label{s:preliminary}

We recall here several properties of convex sets that will be used in the paper.
Given a closed convex set $K\subset\Rn$,
denote by $\Pi_K:\Rn\rightarrow K$ the projection onto $K$.
For all $\vy\in\Rn$, $\Pi_K(\vy)$ is the unique vector in $K$
such that
\begin{equation}
    \label{eq:Pi_characterisation-not-a-cone}
    (\vu -\Pi_K(\vy))^T(\vy-\Pi_K(\vy)) \le 0,
    \qquad
    \vu\in K.
\end{equation}
Inequality \eqref{eq:Pi_characterisation-not-a-cone}
can be rewritten as follows
\begin{equation}
    \scalednorms{\Pi_K(\vy) - \vy}
    +
    \scalednorms{\vu - \Pi_K(\vy)}
    \le
    \scalednorms{\vu - \vy},
    \qquad
    \vy\in\Rn,
    \vu\in K,
    \label{eq:cos}
\end{equation}
which is a consequence of the cosine theorem.
The LS estimator over $K$ is exactly the projection
of $\vy$ onto $K$, i.e., $\ls(K) = \Pi_K(\vy)$.
In this case, \eqref{eq:cos} yields that for all
$\vu\in K$,
\begin{equation}
    \scalednorms{\ls(K) - \vy}
    \le
    \scalednorms{\vu - \vy}
    -
    \scalednorms{\vu - \ls(K)}.
    \label{eq:strong-convexity}
\end{equation}
Inequality
\eqref{eq:strong-convexity}
can be interpreted in terms of strong convexity:
the LS estimator $\ls(K)$ solves an optimization problem
where the function to minimize is strongly convex with respect to the norm
$\scalednorm{\cdot}$. 
Strong convexity grants inequality
\eqref{eq:strong-convexity}, which is stronger than
the inequality
\begin{equation}
    \scalednorms{\ls(\mathcal U) - \vy}
    \le
    \scalednorms{\vu - \vy}
    \qquad
    \text{for all }
    \vu\in\mathcal U
    ,
    \label{eq:ls-nonconvex}
\end{equation}
which holds for any closed set $\mathcal U\subset\R^n$.

Now, assume that $K$ is a closed convex cone.
In this case, \eqref{eq:Pi_characterisation-not-a-cone} implies that
for all $\vy\in\Rn$, $\Pi_K(\vy)$ is the unique vector in $K$ 
such that
\begin{equation}
    \label{eq:Pi_characterisation}
    \Pi_K(\vy)^T \vy = \euclidnorms{\Pi_K(\vy)}
    \quad
    \text{and}
    \quad
    \forall \vtheta\in K,\;\;
    \vtheta^T \vy \le \vtheta^T \Pi_K(\vy).
\end{equation}
The property \eqref{eq:Pi_characterisation} readily implies that
for any $\vv\in\Rn$ we have
\begin{equation}
    \euclidnorm{\Pi_K(\vv)}
    =
    \sup_{\vtheta\in K: \euclidnorm{\vtheta}\le 1}
        \vv^T \vtheta.
    \label{eq:equivalence-delta}
\end{equation}
Define the statistical dimension of the cone $K$ by
\begin{equation}
\label{eq:def-delta}
   \delta(K) 
   \coloneqq
   \E \left[\euclidnorms{\Pi_K(\vg)}\right]
   =
   \E\left[ \vg^T \Pi_K(\vg)\right]
   =
   \E \left[
           \left(
    \sup_{\vtheta\in K: \euclidnorm{\vtheta}\le 1}
        \vg^T \vtheta
    \right)^2
    \right]
   ,
\end{equation}
where $\vg\sim\mathcal N (\vzero,I_{n\times n})$. 
The Gaussian width of a closed convex cone $K$ is defined by 
$w(K) = \E[ \sup_{\vv\in K:\euclidnorm{\vv} = 1} \vg^T\vv]$ where $\vg\sim\mathcal N(\vzero,I_{n\times n})$.
For any closed convex cone $K$, the relation $w^2(K) \le \delta(K) \le w^2(K) + 1$
is established in \cite[Propsition 10.2]{amelunxen2014living}.
The following properties of $\delta(\cdot)$ will be useful for our purpose.
If $K\subset \R^q$, $C\subset \R^p$ are two closed convex cones,
then $K\times C$ is a closed convex cone in $\R^{q+p}$ and
\begin{equation}
    \delta(K\times C) 
    = \delta(K) + \delta(C).
    \label{eq:delta-product}
\end{equation}
The statistical dimension 
$\delta(\cdot)$ is monotone in the following sense:
If $K,L$ are two closed convex cones in $\R^n$ then
\begin{equation}
    \label{eq:delta-inclusion}
    K\subset L
    \quad
    \Rightarrow
    \quad
    \delta(K)
    \le
    \delta(L).
\end{equation}
We refer the reader to \cite[Proposition 3.1]{amelunxen2014living}
for straightforward proofs of the equivalence between
the definitions \eqref{eq:def-delta} and the properties \eqref{eq:delta-product}, \eqref{eq:delta-inclusion} and \eqref{eq:equivalence-delta}.
An exact formula is available for the statistical
dimension of $\increasings$.
Namely, it is proved in
\cite[(D.12)]{amelunxen2014living} that
\begin{equation}
\label{eq:delta-sum1overk}
    \delta(\increasings)
    = \sum_{k=1}^n \frac{1}{k},
\end{equation}
and this formula readily implies that
\begin{equation}
    \log(n)
    \le
    \delta(\increasings)
    \le
    \log(en).
    \label{eq:delta-increasings-logn}
\end{equation}
The following upper bound on the statistical dimension
of the cone $\K^C_{x_1,...,x_n}$ is derived in
\cite{guntuboyina2013global}:
\begin{equation}
\label{eq:bound-delta-54}
    \delta(\K^C_{x_1,...,x_n})
    \le c (\log(en))^{5/4},
\end{equation}
for some constant $c>0$ that depends on the ratio \eqref{eq:ratio}.
In \Cref{thm:delta-convex}, we derive a tighter bound independent of the
design points.

\section{General tools to derive sharp oracle inequalities}
\label{s:tools-soi}
In this section,
we develop two general tools to derive
sharp oracle inequalities
for the LS estimator over a closed convex set.

\subsection{Statistical dimension of the tangent cone}

Let $\vmu\in\Rn$,
let $K$ be a closed convex subset of $\Rn$
and let $\vu\in \R^n$.
Define the tangent cone at $\vu$ by
\begin{equation}
    \mathcal T_{K,\vu} \coloneqq
    \text{closure} \left\{
            t(\vv-\vu):
            \;
            t\ge 0,
            \vv\in K
    \right\}
    \label{eq:def-tangent-cone-TKu}
\end{equation}
If $K$ is a closed convex cone, then $\mathcal T_{K,\vu} = \{\vv - t\vu| \vv\in K, t\ge 0 \}$.
\begin{prop}
    \label{prop:bound-tangent}
    Let $\vmu\in\Rn$,
    let $K$ be a closed convex subset of $\Rn$
    and let $\vu\in K$.
    Then almost surely
    \begin{equation}
        \scalednorms{\ls(K) - \vmu}
        -
        \scalednorms{\vu - \vmu}
        \le
        \frac{\sigma^2}{n} \euclidnorms{\Pi_{\mathcal T_{K,\vu}}(\vg)}
        \label{eq:bound-tangent}
    \end{equation}
    where $\vg = (1/\sigma) \vxi$.
\end{prop}
\begin{proof}
    Let $\hmu=\ls(K)$. 
    Then
    \eqref{eq:strong-convexity} yields
    \begin{equation}
        \euclidnorms{\hmu - \vmu}
        -
        \euclidnorms{\vu - \vmu}
        \le
        2 \vxi^T(\hmu-\vu)
        - \euclidnorms{\hmu-\vu}
        =
        2 \vxi^T\htheta \euclidnorms{\hmu - \vu}
        - \euclidnorms{\hmu-\vu}
        \label{eq:kkt-noise}
    \end{equation}
    where $\htheta$ is defined by $\htheta = (1/\euclidnorm{\hmu - \vu})(\hmu - \vu)$ if $\hmu\ne\vu$
    and $\htheta = \vzero$ otherwise.
    By construction we have $\htheta\in\mathcal T_{K,\vu}$ and $\euclidnorms{\htheta} \le 1$.
    Using the simple inequality $2ab - b^2\le a^2$
    with $a= \sup_{\vtheta\in\mathcal T_{K,\vu}:\euclidnorm{\vtheta}\le 1}\vxi^T\vtheta$
    and $b=\euclidnorm{\hmu - \vu}$,
    we obtain
    \begin{equation}
        \euclidnorms{\hmu - \vmu}
        -
        \euclidnorms{\vu - \vmu}
        \le
        2 \vxi^T(\hmu-\vu)
        - \euclidnorms{\hmu-\vu}
        \le
        \Big(
            \sup_{\vtheta\in\mathcal T_{K,\vu}:\euclidnorms{\vtheta}\le 1}
            \vtheta^T\vxi
        \Big)^2.
        \label{eq:tangent-cone-intermediary}
    \end{equation}
    The equality \eqref{eq:equivalence-delta} completes the proof.
\end{proof}
By definition of the statistical dimension,
$\delta(\mathcal T_{K,\vu}) 
\coloneqq \E  \euclidnorms{\Pi_{\mathcal T_{K,\vu}}(\vg)}$
so that
\eqref{eq:bound-tangent}
readily yields a sharp oracle
inequality in expectation.
Bounds with high probability are obtained as follows.
Let $L\subset \R^n$ be a closed convex cone.
By \eqref{eq:equivalence-delta} we have
$\euclidnorm{\Pi_L(\vg)} = \sup_{\vx\in L:\euclidnorm{\vx}\le 1} \vx^T\vg$.
Thus, by the concentration of suprema of Gaussian processes
\cite[Theorem 5.8]{boucheron2013concentration} 
we have
$$\mathbb P(
    \euclidnorm{\Pi_{L}(\vg)}
    >
    \E \euclidnorm{\Pi_{L}(\vg)}
    + \sqrt{2x}
) \le e^{-x},
$$ and by Jensen's inequality we have
$(\E \euclidnorm{\Pi_{L}(\vg)})^2 \le \delta(L)$.
Combining these two bounds, we obtain
\begin{equation}
    \mathbb P
    (
        \euclidnorm{\Pi_{L}(\vg)}
        \le 
        \delta(L)^{1/2}
        +
        \sqrt{
            2x
        }
    )
    \ge 1 - e^{-x}.
    \label{eq:concentration-delta-not-squared}
\end{equation}
Applying this concentration inequality to the cone $L=\mathcal T_{K,\vu}$ yields the following Corollary.
\begin{cor}
    \label{cor:tangent-cone}
    Let $\vmu\in\Rn$,
    let $K$ be a closed convex subset of $\Rn$,
    let $\vu\in K$ and let $\mathcal T_{K,\vu}$
    be defined in \eqref{eq:def-tangent-cone-TKu}.
    If $\vxi \sim \mathcal N(\vzero, \sigma^2 I_{n\times n})$
    then
    \begin{equation}
        \E\left[ \scalednorms{\ls(K) - \vmu} \right]
        \le
        \scalednorms{\vu - \vmu}
        +
        \frac{\sigma^2}{n}\delta(\mathcal T_{K,\vu})
        .
    \end{equation}
    Furthermore, for all $x>0$ with probability at least $1-e^{-x}$
    we have
    \begin{equation}
        \scalednorms{\ls(K) - \vmu}
        -
        \scalednorms{\vu - \vmu}
        \le
        \frac{\sigma^2}{n}
        \left(
            \delta(\mathcal T_{K,\vu})^{1/2}
            +
            \sqrt{2x}
        \right)^2
        \le
        \frac{\sigma^2}{n}
        \left(
            2 \delta(\mathcal T_{K,\vu})
            +
            4 x
        \right)
        .
    \end{equation}
\end{cor}

In the well-specified case,
a similar upper bound was derived in \cite[Theorem 3.1]{oymak2013sharp}.
\citet{oymak2013sharp}
also proved a worst-case lower bound that matches the upper bound.

The survey \cite{amelunxen2014living}
provides general recipes
to bound from above the statistical dimension of cones
of several types.
For instance, the statistical dimension of $\increasings$
is given by
the exact formula \eqref{eq:delta-sum1overk}.
Bounds on the statistical dimension
of a closed convex cone $K$
can be obtained using metric entropy results, as 
$\sigma^2\delta(K)/n = \E_\vzero\scalednorms{\Pi_K(\vxi)}$
is the risk of the LS estimator $\ls(K)$ when the true
vector is $\vzero$. This technique is used in \cite{guntuboyina2013global}
to derive the bound \eqref{eq:bound-delta-54}.

If $K\subset V$ where $V$ is a subspace of dimension $d_V$, then
by monotonicity of the statistical dimension \eqref{eq:delta-inclusion} we have
$\delta(\mathcal T_{K,\vu} ) \le \delta(V) = d_V$.
In this case, \eqref{eq:bound-tangent} shows
that the constant 4 in
\cite[Proposition 3.1]{tsybakov2014aggregation}
can be reduced to 1.

\subsection{Localized Gaussian widths}
\label{s:uniform-soi}
In this section, we develop 
yet another technique to derive sharp oracle inequalities for LS estimators over closed convex sets.
This technique is associated
with localized Gaussian widths
rather than statistical dimensions of tangent cones.
The result is given in \Cref{thm:isomorphic} below.
Recently, other general methods have been proposed
\cite{chatterjee2013risk,plan2014high,vershynin2014estimation},
but these methods did not provide
oracle inequalities with leading constant 1.

\begin{thm}
    \label{thm:isomorphic}
    Let $K$ be a closed convex subset
    of $\R^n$,
    let $\vmu\in\Rn$.
    Assume that $\vxi\sim \mathcal{N}(0, \sigma^2 I_{n\times n})$
    and
    that for some $\vu\in K$, there exists $t_*(\vu)>0$ such that
    \begin{equation}
        \label{eq:fixed-point-C}
            \E 
            \sup_{\vv\in K:\;\euclidnorm{\vv - \vu}\le t_*(\vu)} \vxi^T \left(\vv - \vu \right)
             \le
             \frac{t_*(\vu)^2}{2}
             .
    \end{equation}
    Then for any $x>0$, with probability greater than $1-e^{-x}$,
    \begin{equation}
        \label{eq:soi-tstar}
        \scalednorms{\ls(K) - \vmu}
        -
        \scalednorms{\vu - \vmu}
        \le
        \frac{(t_*(\vu) + \sigma\sqrt{2x})^2}{n}
        \le
        \frac{2 t_*^2(\vu) + 4\sigma^2x}{n}
        .
    \end{equation}
\end{thm}
The proof of
\Cref{thm:isomorphic} is related to the isomorphic method \cite{bartlett2006empirical}
and the theory of local Rademacher complexities
in regression with random design
\cite{bartlett2005local,koltchinskii2006local}.
\begin{proof}
    Let $t = t_*(\vu)$ and $\hmu=\ls(K)$ for brevity.
    The concentration inequality for suprema of Gaussian processes
    \cite[Theorem 5.8]{boucheron2013concentration} yields
    that on an event $\Omega(x)$ of probability greater than $1-e^{-x}$,
    \begin{equation}
        Z \coloneqq 
        \sup_{\vv\in K: \; \euclidnorm{\vv - \vu}\le t} \vxi^T(\vv-\vu)
        \le
        \E[Z]
        + t \sigma \sqrt{2x}
        \le t^2/2 + t\sigma\sqrt{2x}.
    \end{equation}
    On the one hand, if $\euclidnorm{\hmu - \vu}\le t$, then
    by \eqref{eq:kkt-noise} on $\Omega(x)$ we have
    \begin{equation}
        \euclidnorms{\hmu - \vmu}
            -
        \euclidnorms{\vu - \vmu}
        \le 2\vxi^T(\hmu - \vu) - \euclidnorms{\hmu - \vu}
        \le 2 Z
        \le t^2 + 2t\sigma\sqrt{2x}
        \le (t+\sigma\sqrt{2x})^2.
    \end{equation}
    On the other hand, if $\euclidnorm{\hmu - \vu} > t$, then $\alpha\coloneqq  t / \euclidnorm{\hmu -\vu}$ belongs to $(0,1)$.
    If $\vv = \alpha\hmu + (1-\alpha)\vu$ then
    $
        \alpha(\hmu - \vu) 
        = \vv - \vu
    $, by convexity of $K$ we have $\vv\in K$
    and by definition of $\alpha$ it holds that $\euclidnorm{\vv - \vu} = t$.
    On $\Omega(x)$,
    \begin{align}
        2\vxi^T(\hmu - \vu) - \euclidnorms{\hmu - \vu}
        &= (2/\alpha) \vxi^T(\vv - \vu)
        - t^2/\alpha^2, \\
        &\le (2/\alpha)Z - t^2/\alpha^2
        = (2t/\alpha)(Z/t) 
        - t^2/\alpha^2, \\
        &\le (Z/t)^2
        \le (t+\sigma\sqrt{2x})^2,
    \end{align}
    where we used $2ab - b^2 \le a^2$ 
    with $b=t/\alpha$ and $a = Z/t$.
    Thus \eqref{eq:soi-tstar} holds on $\Omega(x)$
    for both cases $\euclidnorm{\hmu - \vu} \le t$ and
    $\euclidnorm{\hmu - \vu} > t$.
    Finally, inequality $(u+v)^2\le 2u^2 + 2v^2$ yields that
    $(t+\sigma\sqrt{2x})^2\le 2t^2 + 4\sigma^2x$.
\end{proof}

Note that condition \eqref{eq:fixed-point-C} 
does not depend on the true vector $\vmu$,
but only depends on the vector $\vu$ that appears on the right hand side of the oracle inequality.
The left hand side of \eqref{eq:fixed-point-C}
is the Gaussian width of $\C$ localized around $\vu$.
This differs from the recent analysis of \citet{chatterjee2014new}
where 
the Gaussian width localized around $\vmu$ is studied.
An advantage of considering the Gaussian width localized  around $\vu$
is that the resulting oracle inequality \eqref{eq:soi-tstar} is sharp,
i.e., with leading constant 1.
\citet{chatterjee2014new} proved that 
the Gaussian width localized around $\vmu$ characterizes
a deterministic quantity $t_\vmu$ such that $\euclidnorm{\ls(\C) - \vmu}$
concentrates around $t_\vmu$.
This result from \cite{chatterjee2014new}
grants both an upper bound and a lower bound
on $\euclidnorm{\ls(\C) - \vmu}$,
but it does not imply nor is implied by a sharp oracle inequality such as 
\eqref{eq:soi-tstar} above.
Thus, the result of \cite{chatterjee2014new} is
of a different nature than
\eqref{eq:soi-tstar}.

A strategy to find a quantity $t_*$
that satisfies \eqref{eq:fixed-point-C}
is to use metric entropy results together with Dudley integral bound,
although Dudley integral bound 
may not be tight
\cite[Section 13.1, Exercises 13.4 and 13.5]{boucheron2013concentration}.

\section{Sharp bounds in isotonic regression}
\label{s:isotonic}
We study in this section the performance of $\ls(\increasings)$
using the general tools developed in the previous section.
We first apply \Cref{cor:tangent-cone}.
To do so, we need to bound from above
the statistical dimension of the tangent cone
$\mathcal T_{\increasings, \vu}$.
In fact, it is possible to characterize the tangent cone
$\mathcal T_{\increasings, \vu}$
and to obtain a closed formula for its statistical dimension.

\begin{prop}
    Let $\vu\in\increasings$ and let $k=k(\vu)$.
    Let $(T_1,...,T_k)$
    be a partition of $\{1,...,n\}$ such that
    $\vu$ is constant on each $T_j, j=1,...,k$.
    Then
    \begin{equation}
        \mathcal T_{\increasings, \vu}
        =
        \increasingsat_{|T_1|}
        \times
        ...
        \times
        \increasingsat_{|T_k|}.
    \end{equation}
\end{prop}
\begin{proof}
    Let $\mathcal T_{\increasings, \vu} = \mathcal T$ for brevity.
    If $\vu$ is constant, then it is clear that $\mathcal T = \increasings$ so we assume that $\vu$ has at least one jump, i.e., $k(\vu) \ge 2$.
    As $\increasings$ is a cone we have
    $\mathcal T = \{\vv - t\vu| t\ge 0, \vv\in\increasings \}$.
    Thus the inclusion
    $\mathcal T_{\increasings, \vu}
    \subset
    \increasingsat_{|T_1|}
    \times
    ...
    \times
    \increasingsat_{|T_k|}$ is straightforward.
    For the reverse inclusion,
    let $\vx \in \increasingsat_{|T_1|}\times...\times \increasingsat_{|T_k|}$ and
    let $\varepsilon>0$ be the minimal jump of the sequence $\vu$, 
    that is,
    $\varepsilon = \min_{i=1,...,n-1: u_{i+1}>u_i}(u_{i+1} - u_i)$.
    If $t = \inftynorm{\vx} / (4\varepsilon)$ then
    the vector $\vv \coloneqq t\vu + \vx$ belongs to $\increasings$, which completes the proof.
\end{proof}
Using \eqref{eq:delta-product} and \eqref{eq:delta-increasings-logn} we obtain
$\delta(\mathcal T_{\increasings, \vu})
    = \sum_{j=1}^{k(\vu)} \sum_{t=1}^{|T_j|} \frac 1 t
    \le
    \sum_{j=1}^{k(\vu)} \log(e|T_j|)
$.
By Jensen's inequality, this quantity 
is bounded from above by 
$
k(\vu) \log(en / k(\vu))
$.
Applying \Cref{cor:tangent-cone} leads to the following result.

\begin{thm}
    \label{thm:increasings}
    For all $n\ge 2$
    and any $\vmu\in\Rn$,
    \begin{equation}
        \Evmu
        \scalednorms{\hmu^{LS}(\increasings)- \vmu}
        \le
        \min_{\vu\in\increasings}
        \left(
            \scalednorms{\vu - \vmu}
            +
            \frac{\sigma^2 k(\vu)}{n}\log\frac{en}{k(\vu)}
        \right).
        \label{eq:soi-ls}
    \end{equation}
    Furthermore, for any $x>0$ we have
    with probability greater than $1-\exp(-x)$
    \begin{equation}
        \scalednorms{\hmu^{LS}(\increasings)- \vmu}
        \le
        \min_{\vu\in\increasings}
        \left(
            \scalednorms{\vu - \vmu}
            +
            \frac{2\sigma^2 k(\vu)}{n}\log\frac{en}{k(\vu)}
        \right)
        + \frac{4 \sigma^2 x}{n}
        .
        \label{eq:soi-ls-deviation}
    \end{equation}
\end{thm}
Let us discuss some features of \Cref{thm:increasings}
that are new.
First, the estimator $\ls(\increasings)$ satisfies oracle inequalities 
both in deviation 
with exponential probability bounds 
and in expectation, cf. \eqref{eq:soi-ls-deviation} and \eqref{eq:soi-ls}, respectively.
Previously known oracle inequalities for the LS estimator over $\increasings$
were only proved in expectation.

Second,
both \eqref{eq:soi-ls} and \eqref{eq:soi-ls-deviation}
 are sharp oracle inequalities, i.e., with leading constant 1. 
Although sharp oracle inequalities were obtained using aggregation
methods
\cite{bellec2015sharp},
this is the first known sharp oracle inequality for the LS estimator $\ls(\increasings)$.

Third, the assumption $\vmu\in\increasings$ is not needed,
as opposed to the result of \cite{chatterjee2013risk}.

Last,
the constant $1$ in front
of
$\frac{\sigma^2 k(\vu)}{n}\log\frac{en}{k(\vu)}$ in \eqref{eq:soi-ls}
is optimal for the LS estimator.
To see this,
assume that there exists an absolute constant $c<1$ such that
for all $\vmu\in\increasings$
and $\hmu = \ls(\increasings)$,
\begin{equation}
    \Evmu 
    \scalednorms{\hmu - \vmu}
    \le
    \min_{\vu\in\increasings}
    \left(
        \scalednorms{\vu - \vmu}
        +
        \frac{c \sigma^2 k(\vu)}{n}\log\frac{en}{k(\vu)}
    \right).
    \label{eq:constant-c-impossible}
\end{equation}
Set $\vmu=0$.
Thanks to \eqref{eq:delta-increasings-logn},
the left hand side of the above display is
bounded from below by $\sigma^2 \log(n)/n$
while
while the right hand side is equal to $c\sigma^2 \log(en)/n$.
Thus, it is impossible to improve the constant
in front of
$\frac{\sigma^2 k(\vu)}{n}\log\frac{en}{k(\vu)}$ for the estimator $\ls(\increasings)$.
However, it is still possible that for another estimator $\hmu$,
\eqref{eq:constant-c-impossible} holds with $c<1$ or without the logarithmic factor.
We do not know whether such an estimator exists.

We now highlight the adaptive behavior of the estimator
$\ls(\increasings)$.
Let $\vu^*\in\increasings$
be a minimizer
of the right hand side of \eqref{eq:soi-ls}.
Let $k=k(\vu^*)$ and let $(T_1,...,T_k)$ be a partition of $\{1,...,n\}$
such that $\vu^*$ is constant on all $T_j$, $j=1,..,k$.
Given $T_1,...,T_k$,
consider the piecewise constant oracle
\begin{equation}
\hmuora
\in \argmin_{\vu \in W_{T_1,...,T_k}}
\scalednorms{\vy - \vu},
\end{equation}
where $W_{T_1,...,T_k}$ is the linear subspace of all sequences
that are constant on all $T_j$, $j=1,...,k$.
This subspace has dimension $k$,
so the estimator $\hmuora$ satisfies
\begin{equation}
    \Evmu 
    \scalednorms{\hmuora - \vmu}
    =
    \min_{\vu\in W_{T_1,...,T_k}}
    \scalednorms{\vu-\vmu}
    +
    \frac{\sigma^2 k}{n}
    \le
    \scalednorms{\vu^*-\vmu}
    +
    \frac{\sigma^2 k}{n}.
\end{equation}
Thus, \eqref{eq:soi-ls}
can be interpreted in the sense
that
without the knowledge of $T_1,...,T_k$,
the performance of
$\ls(\increasings)$ 
is similar to that of
$\hmuora$
up to the factor $\log(en/k)$.
Of course, the knowledge of $T_1,...,T_k$ is not accessible in practice,
so $\hmuora$ is an oracle that can only serve
as a benchmark.
This adaptive behavior of $\ls(\increasings)$ was observed in \cite{chatterjee2013risk}.

The following results are direct consequences
of \Cref{thm:isomorphic},
Dudley integral bound
and the entropy bounds from \cite{gao2007entropy,chatterjee2014new}.

\begin{cor}
    \label{cor:rate23}
    There exists an absolute constant $c>0$ such that the following holds.
    Let $n\ge 2$ and $\vmu\in\Rn$.
    Assume that $\vxi\sim \mathcal{N}(0, \sigma^2 I_{n\times n})$.
    Then for any $x>0$,
    with probability greater than $1-\exp(-x)$,
    \begin{equation}
        \label{eq:soi-increasing-uniform}
        \scalednorms{
            \ls(\increasings)
            -
            \vmu
        }
        \le
        \min_{\vu\in\increasings}
        \left[
            \scalednorms{
                \vu
                -
                \vmu
            }
            + 2c \sigma^2 \left(\frac{\sigma + V(\vu)}{\sigma n}\right)^{2/3}
        \right]
        + \frac{4\sigma^2 x}{n},
    \end{equation}
    where $V(\cdot)$ is defined in \eqref{eq:def-V}.
\end{cor}

The novelty of 
\Cref{cor:rate23}
is twofold.
First, the leading constant is $1$.
Although model misspecification was considered in \cite{zhang2002risk,guntuboyina2013global},
no oracle inequalities were obtained.
Second, the above sharp oracle inequality holds in deviation,
whereas the previous work derived upper bounds
on the expected squared risk in the well-specified case.
Note that one can derive sharp oracle inequality in expectation
by integration of \eqref{eq:soi-increasing-uniform}.

\section{Convex regression and arbitrary design points}
\label{s:convex}
The goal of this section is to study univariate convex regression
for non-equispaced design points.

\subsection{Parametric rate for any design if $\vmu$ has few affine pieces}
We now present a new argument to bound from above the
statistical dimension of the cone of convex sequences.

\begin{thm}
    \label{thm:delta-convex}
    Let $n\ge 3$. Let $x_1<...<x_n$ be real numbers
    and consider the cone 
    $\K^C_{x_1,...,x_n}$ defined in \eqref{eq:def-convex-x_i}.
    Let
    $\vg\sim\mathcal N (\vzero,I_{n\times n})$.
    Then
    \begin{equation}
        \delta(\K^C_{x_1,...,x_n})
        =
        \E \left[ \euclidnorms{\Pi_{\K^C_{x_1,...,x_n}}(\vg)} \right]
        \le 8 \log(en).
        \label{eq:bound-delta-convex}
    \end{equation}
\end{thm}
\begin{proof}
    Let $K = \K^C_{x_1,...,x_n}$ for brevity.
    A convex sequence $\vu=(u_1,...,u_n)\in K$ is first non-increasing and then nondecreasing,
    that is,
    there exists $m\in\{1,...,n\}$ such that
    $u_1\ge u_2 \ge ... \ge u_m \le u_{m+1} \le ... \le u_n$,
    hence the sequence $\vu$ is unimodal.
    Thus, if $K_m,m=1,...,n$
    are the sets defined in \eqref{eq:def-Km-unimodal},
    then $K \subset \mathcal U = \cup_{m=1,...,n} K_m $.
    Using \eqref{eq:delta-product}, \eqref{eq:delta-inclusion} and \eqref{eq:delta-increasings-logn} 
    we obtain 
    \begin{equation}
        \delta(K_m) \le \delta(\mathcal S^\downarrow_m \times \mathcal S^\uparrow_{n-m}) \le  \log(em) + \log(e(n-m)) \le 2 \log(en).
    \end{equation}
    By \eqref{eq:equivalence-delta}, almost surely we have
    \begin{equation}
        0
        \le
    \euclidnorm{\Pi_K(\vg)}
    =
    \sup_{\vu\in K: \euclidnorm{\vu}\le 1}
        \vg^T \vu
    \le
    \max_{m=1,...,n}
    \sup_{\vu\in K_m: \euclidnorm{\vu}\le 1}
        \vg^T \vu
    =
    \max_{m=1,...,n}
    \euclidnorm{\Pi_{K_m}(\vg)}
    .
    \end{equation}
    Using \eqref{eq:concentration-delta-not-squared}
    and the union bound,
    for all $x>0$,
    we have with probability at least $1-e^{-x}$ the inequality
    $
    \euclidnorms{\Pi_K(\vg)}
    \le 
    \max_{m=1,...,n}
    \delta(K_m)^{1/2}+\sqrt{2(x+\log n)}$.
    As $(a+b)^2\le 2a^2 + 2b^2$, on the same event of probability at least $1-e^{-x}$
    we have
    \begin{equation}
        \euclidnorms{\Pi_K(\vg)}
        \le
        2 \max_{m=1,...,n}
        \delta(K_m)
        + 4 (x+\log n)
        \le 4 \log(en) + 4(x + \log n)
        .
    \end{equation}
    Integration of this probability bound completes the proof.
\end{proof}

Remarkably, this bound on the statistical dimension
does not depend on the design points $x_1,...,x_n$.
Furthermore, the bound 
\eqref{eq:bound-delta-convex}
improves 
upon
\eqref{eq:bound-delta-54}
as the exponent $5/4$ is reduced to 1.
\eqref{eq:bound-delta-convex}

\begin{prop}
    \label{prop:tangent-cone-convex}
    Let $n\ge 3$, and let $\vu$
    be an element of the cone
    $\K^C_{x_1,...,x_n}$ defined in \eqref{eq:def-convex-x_i}.
    The statistical dimension of the tangent cone
    at $\vu$ satisfies
    \begin{equation}
        \delta(\mathcal T_{\K^C_{x_1,...,x_n},\vu})
        \le
        8 q(\vu) \log\left(\frac{en}{q(\vu)}\right).
    \end{equation}
\end{prop}
\begin{proof}
    Let $q=q(\vu)$.
    Let $(T_1,...,T_q)$ be a partition of $\{1,...,n\}$
    such that $\vu$ is affine on each $T_j,j=1,...,q$.
    Let $\vx\in\K^C_{x_1,...,x_n}$.
    A convex sequence minus an affine sequence is convex, thus
    for all $j=1,...,q$, $(\vx - \vu)_{T_j}$
    is convex in the sense that it belongs to
    $\K^C_{x_i:i\in T_j}$. Thus
    \begin{equation}
        \mathcal T_{\K^C_{x_1,...,x_n},\vu}
        \subset
        \C \coloneqq 
        \K^C_{x_i:i\in T_1}
        \times
        \K^C_{x_i:i\in T_2}
        \times
        ...
        \times
        \K^C_{x_i:i\in T_q}.
    \end{equation}
    Using \eqref{eq:delta-inclusion},
    \eqref{eq:delta-product}, \Cref{thm:delta-convex} and Jensen's inequality we have
    \begin{equation}
        \delta\left(
        \mathcal T_{\K^C_{x_1,...,x_n},\vu}
        \right)
        \le
        \delta(\C)
        \le \sum_{j=1}^q 8 \log(e|T_j|)
        \le 
        8 q \log\left(\frac e q \sum_{j=1}^q |T_j|\right)
        =
        8 q \log\left(\frac{en}{q} \right).
    \end{equation}
\end{proof}
Combining \Cref{cor:tangent-cone} and \Cref{prop:tangent-cone-convex}
yields the following.
\begin{thm}
    \label{thm:soi-for-convex}
    Let $n\ge 3$
    and $\vmu\in\Rn$.
    Let $x_1<...<x_n$ be real numbers.
    Then for any $x>0$, the estimator $\hmu=\hmu^{LS}(\K^C_{x_1,...,x_n})$
    satisfies
    \begin{equation}
        \scalednorms{\hmu - \vmu}
        \le
        \min_{\vu\in\K^C_{x_1,...,x_n}}
        \left(
            \scalednorms{\vu - \vmu}
            +
            \frac{ 16 \sigma^2 q(\vu)}{n}\log\frac{en}{q(\vu)}
        \right)
        + \frac{4 \sigma^2 x}{n}
    \end{equation}
    with probability greater than $1-\exp(-x)$.
\end{thm}
\Cref{thm:soi-for-convex}
does not depend on the design points $x_1,...,x_n$.
In particular, \Cref{thm:soi-for-convex}
and the corresponding result in expectation
hold for non-equispaced design points and
design points that can be arbitrarily close to each other.
This improves upon the oracle inequality
\eqref{eq:convexs-adapt-example-oi}
proved in \cite{guntuboyina2013global,chatterjee2013risk}
where $C$ is strictly greater than 1 and 
depends on the design points through the ratio
\eqref{eq:ratio}.
Thanks to \Cref{cor:tangent-cone},
these oracle inequalities
hold
in deviation
with exponential probability bounds
and in expectation for any $\vmu\in\Rn$,
whereas previously known oracle inequalities
from \cite{guntuboyina2013global,chatterjee2013risk}
only hold in expectation
under the additional assumption that $\vmu\in\K^C_{x_1,...,x_n}$.

\subsection{Worst-case design points in convex regression and the rate $n^{-2/3}$}
\label{ss:convex-unimodal}
The nonparametric rate for estimation of convex sequences
is of order $n^{-4/5}$ for equispaced design points.
This was established in \cite{guntuboyina2013global}
using metric entropy bounds.
The following result combines the metric entropy bounds from \cite{guntuboyina2013global} with \Cref{thm:isomorphic}.

\begin{cor}
    \label{cor:rate45}
    There exist absolute constants $\kappa, C>0$ such that the following holds.
    Let $n\ge 3$ and $\vmu\in\Rn$.
    Assume that $\vxi\sim \mathcal{N}(0, \sigma^2 I_{n\times n})$.
    Then for any $x>0$,
    with probability greater than $1-\exp(-x)$,
    \begin{equation}
        \scalednorms{
            \ls(\convexs)
            -
            \vmu
        }
        \le
        \min_{\vu\in\convexs}
        \left[
            \scalednorms{
                \vu
                -
                \vmu
            }
            + \frac{C \left(R_{\vu} \sigma^4\right)^{2/5}   \log(en)   }{n^{4/5}}
        \right]
        + \frac{16\sigma^2 x}{n}.
    \end{equation}
    where $R_{\vu} = \max(\sigma, \min(\{ \scalednorm{\vu - \tau}, \tau \in\Rn \text{ and } \tau \text{ is affine} \}))$,
    and the minimum on the right hand side is taken over all $\vu\in\convexs$ such that
    \begin{equation}
    \label{eq:rate45condition}
        n R_{\vu}^2 \ge \kappa \log(en)^{5/4},
    \end{equation}
\end{cor}
Thanks to the metric entropy bounds of \cite{guntuboyina2013global},
\Cref{cor:rate45} holds for equispaced design points
or design points that are almost equispaced in the sense
that the ratio \eqref{eq:ratio} is bounded from above by a numerical
constant.
It is natural to ask
whether the nonparametric rate $n^{-4/5}$ can be achieved
by the LS estimator for any design points.
The following result provides a negative answer:
There exist design points such that no estimator can
achieve a better rate than $n^{-2/3}$.
This rate $n^{-2/3}$ is substantially slower than the nonparametric rate $n^{-4/5}$
achieved by the LS estimator in convex regression with equispaced design points.

\begin{thm}
    \label{thm:convex-23-lower}
    Let $V>0$.
    There exists design points $x_1<...<x_n$ that depend on $V$
    such that for any estimator $\hmu$,
    \begin{equation}
        \sup_{\vmu\in\K^C_{x_1,...,x_n}\cap\increasings:\;\; \mu_n - \mu_1 \le 2 V}
        \mathbb P_\vmu
        \left(
            \scalednorms{\hmu - \vmu}
            \ge
            C \sigma^2 \left(\frac{V^2}{\sigma^2 n^2 }\right)^{1/3}
        \right) \ge c,
    \end{equation}
    where $c,C>0$ are absolute constants.
\end{thm}

The intuition behind this result is the following.
Let $\vmu=(\mu_1,...,\mu_n)^T\in\increasings$ be strictly increasing
and let $\epsilon \coloneqq \frac 1 2 \wedge \min_{i=2,...,n-1}\frac{\mu_{i+1} - \mu_i}{\mu_i - \mu_{i-1}}$.
If we define the design points $x_1<...<x_n$ by 
$x_i = - \epsilon^i$ for all $i=1,...,n$
then
$\vmu\in K^C_{x_1,...,x_n}$ (this statement is made rigorous in the proof of \Cref{thm:convex-23-lower}).
That is, for any strictly increasing sequence $\vmu$,
there are geometrically spaced design points $x_1<...<x_n$ such that
$\vmu=(f(x_1),...,f(x_n))^T$ for some convex function $f$.
As explained in the following proof,
this observation yields that 
the minimax lower bound over $\increasings$ implies a minimax lower bound over $\K^C_{x_1,...,x_n}$
for a specific choice of design points $x_1,...,x_n$.

\begin{proof}[Proof of \Cref{thm:convex-23-lower}]
    For any $\vu\in\increasings$, let $V(\vu) = u_n - u_1$.
    It was proved  in \cite[Proposition 4 and Corollary 5]{bellec2015sharp}
    that for some integer $M\ge 2$, there exist $\vmu_0,...,\vmu_M\in\increasings$
    such that $V(\hmu_j) \le V$ and
    \begin{equation}
        \scalednorm{\vmu_j - \vmu_k}
        \ge
        \frac{C \sigma^2}{4} \left(\frac{V^2}{\sigma^2 n^2 }\right)^{1/3},
        \qquad
        \frac{n}{2\sigma^2}\scalednorm{\vmu_j - \vmu_0} \le \frac{\log M}{16}
        \label{eq:relations-kullback}
    \end{equation}
    for all distinct $j,k\in\{0,...,M\}$ and some absolute constant $C>0$.
    The quantity $\frac{n}{2\sigma^2}\scalednorm{\vmu_j - \vmu_0}$
    is Kullback-Leibler
    divergence from $\mathcal N(\vmu_j,\sigma^2 I_{n\times n})$
    to $\mathcal N(\vmu_0,\sigma^2 I_{n\times n})$.

    Define $\vv=(v_1,...,v_n)^T$ by $v_i = iV/n$ for all $i=1,...,n$ so that $V(\vv) \le V$ and $\vv$
    is strictly increasing.
    We define $\vu^0,...,\vu^M$ by $\vu^j = \vmu_j + \vv$ so that
    $\vu^0,...,\vu^M$ are strictly increasing.
    Furthermore, since $\vmu_j - \vmu_k = \vu^j - \vu^k$
    it is clear that \eqref{eq:relations-kullback} still holds if $\vmu_j,\vmu_k$
    are replaced by $\vu^j,\vu^k$.
    Applying \cite[Theorem 2.7]{tsybakov2009introduction} yields that
    for any estimator $\hmu$,
    \begin{equation}
        \sup_{j=0,...,M}
        \mathbb P_{\vu_j}
        \left(
            \scalednorms{\hmu - \vmu}
            \ge
            C \sigma^2 \left(\frac{V^2}{\sigma^2 n^2 }\right)^{1/3}
        \right) \ge c,
    \end{equation}
    where $c,C>0$ are absolute constants.

    Let $\epsilon \coloneqq \frac 1 2 \wedge \min_{j=1,...,M} \min_{i=2,...,n-1} \frac{u^j_{i+1} - u^j_i}{u^j_i - u^j_{i+1}}$.
    Since the sequences $\vu^0,...,\vu^M$ are strictly increasing we have $\epsilon>0$.
    Define the design points $x_1<...<x_n$ by $x_i = - \epsilon^i$ for all $i=1,...,n$.
    Then for all $j=0,...,M$ 
    we have
    \begin{equation}
        \frac{x_{i+1} - x_i}{x_i - x_{i-1}}
        = 
        \epsilon \le \frac{u^j_{i+1} - u^j_i}{u^j_i - u^j_{i+1}}
    \end{equation}
    for all $i=2,...,n-1$, and by \eqref{eq:def-convex-x_i}
    this implies that $\vu^j\in\K^C_{x_1,...,x_n}$.
    It remains to show that $V(\vu^j) \le 2V$, which is a consequence of $V(\vmu_j) \le V$ and $V(\vv)\le V$.
\end{proof}

If the practitioner can choose the design points, 
then geometrically spaced design points should be avoided.

Any convex function is unimodal so that the inclusion
$\K^C_{x_1,...,x_n}\subset\mathcal U$ holds for any design points $x_1<...<x_n$.
Intuitively, this inclusion means that convexity brings more structure than unimodality.
\Cref{thm:convex-unimodal} below shows that the convex LS
enjoys essentially the same risk bounds and oracle inequalities
those satisfied by the unimodal LS estimator
in \Cref{thm:unimodal-soi,thm:unimodal-soi-n23}.

\begin{thm}
    \label{thm:convex-unimodal}
    Let $\vmu\in \R^n$
    and let $x_1<...<x_n$ be any real numbers.
    Then for all $x>0$, with probability at least $1-2 e^{-x}$, the estimator
    $\hmu = \ls(\K^C_{x_1,...,x_n})$ satisfies
    \begin{equation}
    \scalednorm{\hmu - \vmu}
        \le
        \min_{\vu\in\mathcal U}\left[
            \scalednorm{\vu - \vmu}
            \vee
        \frac{\sigma}{\sqrt n} 
            \left(
                2\sqrt{(k(\vu)+1)\log(\tfrac{en}{k(\vu)+1})} 
                +
                3\sigma \sqrt{2(x+\log n)}
            \right)
        \right]
    .
    \end{equation}
\end{thm}

\begin{thm}
    \label{thm:convex-rate23}
    There exists an absolute constant $c>0$ such that the following holds.
    Let $\vmu\in \R^n$
    and let $x_1<...<x_n$ be any real numbers.
    Then for all $x>0$, with probability at least $1-2 e^{-x}$, the estimator
    $\hmu = \ls(\K^C_{x_1,...,x_n})$ satisfies
    \begin{align}
        \scalednorm{\hmu  - \vmu}
        \le 
        \min_{\vu\in\mathcal U}
        \left[
            \scalednorm{\vu - \vmu}
            +
            2 c \sigma \left(\frac{\sigma +V(\vu)}{\sigma n}\right)^{1/3}
        \right]
        + \frac{2 (2+\sqrt 2) \sigma \sqrt{x+\log(en)}}{\sqrt n}
        .
    \end{align}
\end{thm}

The proofs of \Cref{thm:convex-unimodal,thm:convex-rate23} are given in \Cref{s:proofs-convex-unimodal}.
\Cref{thm:convex-rate23} shows that the convex LS estimator achieves a rate
of order $n^{-2/3}$ for any design points.
Together, \Cref{thm:convex-rate23,thm:convex-23-lower}
establish that this rate is minimal over all design points and all univariate convex functions with bounded total variation.
To make this precise, define the minimax quantity
\begin{equation}
    \mathfrak R (V)
    \coloneqq
    \sup_{x_1<...<x_n}
    \inf_{\hmu}
    \sup_{\vmu\in \K^C_{x_1,...,x_n}: \; V(\vmu) \le V}
    \Evmu [
        \scalednorms{\hmu - \vmu}
    ],
    \qquad
    V 
    >
    0,
\end{equation}
where the first supremum is taken over all $x_1,...,x_n\in\R$ such that $x_1<...<x_n$
and the infimum is taken over all estimators that may depend on $x_1,...,x_n$
(for instance, the convex LS estimator $\ls(\K^C_{x_1,...,x_n}))$ depends on the design points).
The quantity $\mathfrak R(V)$ represents the minimax risk
over all possible univariate design points,
and over all convex sequences.
By taking $\vu = \vmu$ in \Cref{thm:convex-rate23} and by integration,
we obtain that 
\begin{equation}
    \sup_{x_1<...<x_n}
    \sup_{\vmu\in \K^C_{x_1,...,x_n}: \; V(\vmu) \le V}
    \Evmu [
        \scalednorms{\ls(\K^C_{x_1,...,x_n}) - \vmu}
    ]
    \le C \sigma^2 \left(\frac{\sigma + V}{\sigma n}\right)^{2/3}
\end{equation}
for some absolute constant $C>0$.
On the other hand, \Cref{thm:convex-23-lower} and Markov inequality
yield that $\mathfrak R(V) \ge C' \sigma^2 (\frac{V}{\sigma n})^{2/3}$
for some absolute constant $C'>0$.
In summary,
\begin{equation}
    C' \sigma^2 \left(\frac{V}{\sigma n}\right)^{2/3}
    \le \mathfrak R(V)
    \le C \sigma^2 \left(\frac{\sigma + V}{\sigma n}\right)^{2/3}.
\end{equation}
This establishes that the nonparametric rate of univariate convex regression
over all possible design points is of order $n^{-2/3}$ provided that $V\ge \sigma$.
This rate is substantially slower than the rate $n^{-4/5}$
observed by \citet{guntuboyina2013global} for equispaced design points.
In summary, there is no hope to achieve the nonparametric rate $n^{-4/5}$ 
for any univariate design points.

As a convex function is unimodal, the inclusion $\K^C_{x_1,...,x_n}\subset \mathcal U$
holds.
The convex constraints that define $\K^C_{x_1,...,x_n}$
are more restrictive than the unimodal constraint,
i.e., convexity brings more structure than unimodality.
For equispaced design points, the extra structure brought by convexity
yields a nonparametric rate of order $n^{-4/5}$
which is faster than the unimodal nonparametric rate $n^{-2/3}$.
However, for some worst-case design points, this extra structure is uninformative
from a statistical standpoint:
The nonparametric rates of convex and unimodal regression are of the same order $n^{-2/3}$.

\section{Estimation of the projection of the true parameter}
\label{s:misspecification}
Let $K$ be a subset of $\Rn$.
If the unknown regression vector $\vmu$ lies in $K$,
we say that the model is well-specified.
If $\vmu\in K$,
an estimator $\hmu$ enjoys good performance if
the squared error 
\begin{equation}
\label{eq:loss}
    \scalednorms{\hmu-\vmu}
\end{equation}
is small, either in expectation or with high probability.
If $\vmu\notin K$, we say that the model is misspecified.
In that case, several natural quantities are
of interest to assess the performance
of an estimator $\hmu$. 
The regret of order 1 and the regret of order 2
of an estimator $\hmu$ are defined by
\begin{equation}
    \label{eq:regret}
    R_2(\hmu) \coloneqq \scalednorms{\hmu - \vmu}
    - \min_{\vu\in K}\scalednorms{\vu-\vmu},
    \qquad
    R_1(\hmu) \coloneqq \scalednorm{\hmu - \vmu}
    - \min_{\vu\in K}\scalednorm{\vu-\vmu}.
\end{equation}
If $\hmu\in K$, it is clear that $R_1(\hmu) \ge 0,R_2(\hmu)\ge 0$
and that
$R_1(\hmu)^2 \le R_2(\hmu)$
by using the basic inequality $(a-b)^2\le |a^2 - b^2|$ for all $a,b\ge 0$.

If the set $K$ is closed and convex and $\hmu$ is valued in $K$,
another quantity of interest is the estimation error
with respect to the projection of $\vmu$ onto $K$:
\begin{equation}
    \label{eq:estimation-error}
    \scalednorms{\hmu - \Pi_K(\vmu)}.
\end{equation}
The triangle inequality yields
$
|\scalednorm{\hmu - \vmu}
    - \min_{\vu\in K}\scalednorm{\vu-\vmu}|
\le
\scalednorm{\hmu - \Pi_K(\vmu)}$.
Furthermore, if $\hmu$ is valued in $K$ and $K$ is convex,
then \eqref{eq:cos} with $\vy$ replaced by $\vmu$ and $\vu$ replaced by $\hmu$
can be rewritten as
\begin{equation}
    \scalednorms{\hmu - \Pi_K(\vmu)}
    \le
    \scalednorms{\hmu - \vmu}
    - \scalednorms{\vmu - \Pi_K(\vmu)},
    \quad
    \text{ for all }
    \hmu\in K.
\end{equation}
Thus, if $K$ is convex, for any estimator $\hmu$
valued in $K$ we have
$R_1^2(\hmu) \le \scalednorms{\hmu - \Pi_E(\vmu)} \le R_2(\hmu)$.
The following Proposition sums up the relationship
between the quantity \eqref{eq:estimation-error} and the regrets of order 1 and 2
in the case of a closed convex set $K$.
\begin{prop}[Misspecification inequalities]
    \label{prop:misspecification}
    Let $\vmu\in\R^n$ and let $K\subset\R^n$ be a closed convex set.
    Then $\min_{\vu\in K}\scalednorm{\vu-\vmu} = \scalednorm{\Pi_K(\vmu) - \vmu}$
    and
    for any $\hmu\in K$, the following holds almost surely
    \begin{equation}
    \left(
        \scalednorm{\hmu - \vmu}
        -
        \scalednorm{\Pi_K(\vmu) - \vmu}
    \right)^2
    \le
    \scalednorms{\hmu - \Pi_K(\vmu)}
    \le
    \scalednorms{\hmu - \vmu}
    - \scalednorms{\Pi_K(\vmu) - \vmu}.
    \end{equation}
\end{prop}
Estimation of $\Pi_K(\vmu)$ by the LS estimator
$\ls(K)$ has been considered in \cite[Section 4]{zhang2002risk}
for $K=\increasings$, and 
in \cite[Section 6]{guntuboyina2013global} for $K=\convexs$.
\Cref{prop:misspecification} above shows that for any
quantity $r(K)$ and any estimator $\hmu$ valued in a closed convex set $K$, we have
\begin{equation}
    \scalednorms{\hmu - \vmu}
    \le
    \scalednorms{\Pi_K(\vmu) - \vmu}
    + r(K)
    \quad
    \text{implies}
    \quad
    \scalednorms{\hmu - \Pi_K(\vmu)} ] \le r(K),
\end{equation}
i.e., a sharp oracle inequality with leading constant 1 automatically 
implies an upper bound on the estimation error with respect to the projection of $\vmu$ onto $K$.
For instance, if $K=\increasings$, \Cref{thm:increasings}
and \Cref{cor:rate23} with $\vu=\Pi_K(\vmu)$ imply the following.
If $\hmu = \ls(\increasings)$ and $\vpi = \Pi_{\increasings}(\vmu)$
then
\begin{align}
    &\mathbb P \left(
        \scalednorms{\hmu - \vpi}
        \le
        \frac{\sigma^2 k(\vpi)}{n} \log\frac{en}{k(\vpi)}
    \right) \ge 1 - e^{-x},
    \\
    &\mathbb P \left(
        \scalednorms{\hmu - \vpi}
        \le
        2c \sigma^2 \left(
            \frac{\sigma + V(\vpi)}{\sigma n}\right
        )^{2/3}
        + \frac{4\sigma^2 x}{n} 
    \right) \ge 1 - e^{-x},
\end{align}
for any $\vmu\in\R^n$,
where $c>0$ is an absolute constant and $V(\cdot)$ is the total variation defined in \eqref{eq:def-V}.
That is, in the misspecified case, the LS estimator $\ls(\increasings
)$ estimates $\vpi$ at the parametric rate
if $\vpi$ has few constant pieces,
and at the nonparametric rate $n^{-2/3}$ otherwise.
Similar conclusions can be drawn in convex regression
from \Cref{thm:soi-for-convex} and \Cref{cor:rate45}.
If $\hmu = \ls(\convexs)$
and $\vpi = \Pi_{\convexs}(\vmu)$
we have
\begin{align}
    &\mathbb P \left(
        \scalednorms{\hmu - \vpi}
        \le
            \frac{ 16 \sigma^2 q(\vpi)}{n}\log\frac{en}{q(\vpi)}
            + \frac{4 \sigma^2 x}{n}
    \right) \ge 1 - e^{-x}, \\
    &\mathbb P\left(
        \scalednorms{\hmu - \vpi}
        \le
            \frac{C \left(R_{\vpi} \sigma^4\right)^{2/5}   \log(en)   }{n^{4/5}}
            + \frac{16\sigma^2 x}{n}
    \right) \ge 1 - e^{-x},
\end{align}
provided that \eqref{eq:rate45condition} holds with $\vu$ replaced by $\vpi$.

Finally,
the following Corollary is an outcome of \Cref{prop:misspecification},
\Cref{prop:bound-tangent},
and \Cref{thm:isomorphic}.
\begin{cor}
    Let $K$ be a closed convex set and let $\vmu\in\R^n$.
    Then almost surely,
    $\scalednorm{\ls(K) - \Pi_K(\vmu)} \le (\sigma/\sqrt n) \euclidnorm{\Pi_{\mathcal T_{K, \Pi_K(\vmu)}}(\vg)}$
    where $\vg = (1/\sigma) \vxi$.
    Furthermore, if $t_*>0$ is such that
    \begin{equation}
            \E 
            \sup_{\vv\in K:\;\euclidnorm{\vv - \Pi_K(\vmu)}\le t_*} \vxi^T \left(\vv - \Pi_K(\vmu) \right)
             \le
             \frac{t_*^2}{2},
    \end{equation}
    then for all $x>0$, with probability at least $1-e^{-x}$ we have
    \begin{equation}
        \euclidnorm{\ls(K) - \Pi_K(\vmu)}
        \le (
            t_* + \sigma \sqrt{2x}
        )
        .
    \end{equation}
\end{cor}
These results highlight a major advantage of oracle inequalities with leading
constant 1 over oracle inequalities with leading constant strictly greater than
1 such that \eqref{eq:increasings-adapt-example-oi}.
Indeed, oracle inequalities with leading constant 1 yield an upper bound on the estimation error
$\scalednorm{\ls(K) - \Pi_K(\vmu)}$ for any closed convex set $K$.

\appendix

\section{From unimodal to convex regression}
\label{s:proofs-unimodal-to-convex}
In this section, we develop the tools needed to study unimodal regression
and to prove \Cref{thm:convex-unimodal,thm:convex-rate23}
in convex regression.
\Cref{thm:convex-unimodal,thm:convex-rate23} are similar to
\Cref{thm:unimodal-soi,thm:unimodal-soi-n23}
in unimodal regression. Their proofs share the same fundamental argument.

\subsection{Sharp oracle inequalities in unimodal regression}
\label{s:unimodal}

The following result improves
the risk bound \eqref{eq:unimodal-32}
by reducing the exponent $3/2$ to $1$,
proving that the lower bound 
\eqref{eq:minimax-lower-U} is actually tight.

\begin{thm}
    \label{thm:unimodal}
    Let $\vmu\in\mathcal U$
    and let $k= k(\vmu)$.
    If $\vxi\sim\mathcal N(\vzero,\sigma^2 I_{n\times n})$
    then for all $x$,
    \begin{equation}
        \scalednorm{\ls(\mathcal U) - \vmu}
        \le
        \frac{2\sigma}{\sqrt n} \left(
            \sqrt{(k+1) \log(\tfrac{en}{k+1}) }
            +
            \sqrt{2(x + \log n)}
        \right)
    \end{equation}
    holds with probability at least $1-e^{-x}$.
\end{thm}

Let $\mathcal T_{m,\vu}$ be the tangent cone of $K_m$ at some $\vu\in\mathcal U$,
that is,
\begin{equation}
    \mathcal T_{m, \vu}
    \coloneqq
    \text{closure}
    \{
        t(\vx - \vu)|
        t\ge 0,
        \vx\in K_m
    \}.
\end{equation}
For any $\vu\in\mathcal U$,
define the set $\mathcal W_\vu$ and the random variable $Y_\vu$
by
\begin{equation}
    \mathcal W_\vu\coloneqq \cup_{m=1,...,n}\mathcal T_{m,\vu},
    \qquad
    Y_\vu\coloneqq 
    \sup_{\vv\in\mathcal W_\vu: \euclidnorm{\vv}\le 1} \vxi^T\vv.
    \label{eq:def-Y-u}
\end{equation}
Our proof of \Cref{thm:unimodal} relies on an upper bound on the statistical dimension of the above tangent cones
and the concentration of of the random variable $Y_\vu$.

\begin{lemma}
    \label{lemma:Y_vu}
    Let $\vu\in\mathcal U$ and $x>0$.
    With probability at least $1-e^{-x}$ we have
    $
        Y_\vu
        \le
        \sigma \max_{m=1,...,n}\delta(\mathcal T_{m,\vu})
        +
        \sigma \sqrt{2(x + \log n)}
    $.
\end{lemma}
\begin{lemma}
    \label{lemma:T_m}
    If $\vu\in\mathcal U$
    then
    $\max_{m=1,...,n}\delta(\mathcal T_{m,\vu}) \le (k(\vu)+1) \log(\frac{en}{k(\vu)+1})$.
\end{lemma}

\Cref{lemma:Y_vu} and \Cref{lemma:T_m} are proved in \Cref{s:lemma-T_m} below.
\Cref{lemma:Y_vu} is proved using
the union bound and 
\eqref{eq:concentration-delta-not-squared} applied to the cones $\mathcal T_{1,\vu},...,\mathcal T_{n,\vu}$,
while \Cref{lemma:T_m} is a straightforward consequence of 
\eqref{eq:delta-product}, \eqref{eq:delta-inclusion}
and \eqref{eq:delta-increasings-logn}.
The two Lemmas above yield that
\begin{equation}
    \mathbb P
    \left(
        Y_\vu
        \le
        \sigma\sqrt{(k(\vu)+1) \log(\tfrac{en}{k(\vu)+1}) }
        +
        \sigma \sqrt{2(x + \log n)}
    \right) \ge 1 -e^{-x}.
    \label{eq:concentration-Y-u}
\end{equation}
We a now ready to prove \Cref{thm:unimodal}.

\begin{proof}[Proof of \Cref{thm:unimodal}]
    Let $\hmu = \ls(\mathcal U)$ and $\hat R = \euclidnorm{\hmu - \vmu}$.
    Inequality \eqref{eq:ls-nonconvex} with $\vu=\vmu$ can be rewritten as
    $\hat R^2 \le 2 \vxi^T(\hmu - \vmu)$.
    It is clear that $(1/\hat R)(\hmu - \vmu)$ has norm 1 and belongs
    to $\mathcal V_* \coloneqq \cup_{m=1,...,n}\mathcal T_{m,\vmu}$. Thus
    \begin{equation}
        \hat R
        = \frac{\hat R^2}{\hat R}
        \le
        \frac{2\vxi^T(\hmu - \vmu)}{\hat R}
        \le
        2
        \sup_{\vv\in\mathcal W_\vmu: \euclidnorm{\vv}\le 1} \vxi^T\vv
        \le 2 Y_\vmu.
    \end{equation}
    The concentration inequality 
    \eqref{eq:concentration-Y-u}
    completes the proof.
\end{proof}

During the writing of the revision of the present article in which \Cref{thm:unimodal} was introduced,
we became aware of a similar result by
\citet{flammation2016seriation}
obtained independently
in the context of statistical seriation.
Interestingly, \Cref{thm:unimodal} and the result of \citet{flammation2016seriation} are proved using different techniques.
\Cref{thm:unimodal} is an outcome of the concentration inequality 
\eqref{eq:concentration-delta-not-squared}
and of upper bounds on the statistical dimension of tangent cones,
while \citet{flammation2016seriation} prove an oracle inequality
using metric entropy bounds and the variational representation studied in \cite{chatterjee2014new,chatterjee2015adaptive}.
An advantage of the proof presented above is that the numerical constants
of \Cref{thm:unimodal} are explicit and reasonably small.

The proof of \Cref{thm:unimodal} can be slightly modified to yield an oracle inequality.
The following Theorems recover the results of \citet{flammation2016seriation}.

\begin{thm}
    \label{thm:unimodal-soi}
    Let $\vmu\in\R^n$.
    Furthermore, 
    if $\vxi\sim\mathcal N(\vzero,\sigma^2 I_{n\times n})$
    then for all $x>0$, we have
\begin{equation}
    \scalednorm{\ls(\mathcal U) - \vmu}
    \le
    \min_{\vu\in\mathcal U}
    \left[
        \scalednorm{\vu - \vmu}
        +
        \frac{2\sigma}{\sqrt n} \left(
            \sqrt{(k(\vu)+1) \log(\tfrac{en}{k(\vu)+1}) }
            +
            \sqrt{2(x + \log n)}
        \right)
    \right]
\end{equation}
    with probability at least $1-e^{-x}$.
\end{thm}

\begin{proof}[Proof of \Cref{thm:unimodal-soi}]
    Let $\hmu = \ls(\mathcal U)$ and let $\vu$ be a minimizer of the right hand side.
    We first prove that almost surely,
    \begin{equation}
        \scalednorm{\ls(\mathcal U) - \vmu}
        \le
        \min_{\vu\in\mathcal U}
        \left[
            \scalednorm{\vu - \vmu}
            +
            \frac{2 Y_\vu}{\sqrt n}
        \right].
        \label{eq:oi-unimodal-Y-u}
    \end{equation}
    Let $\hat R \coloneqq \euclidnorm{\hmu - \vmu}$
    and $R \coloneqq \euclidnorm{\vu - \vmu}$.
    The vector $\vtheta\coloneqq \frac{\hmu - \vu}{\hat R + R}$
    belongs to the set $\mathcal W_\vu$ defined in \eqref{eq:def-Y-u}
    and $\vtheta$  has norm at most one because of the triangle inequality
    $\euclidnorm{\hmu - \vu} \le \hat R + R$.
    Inequality \eqref{eq:ls-nonconvex} can be rewritten as $\hat R^2 - R^2 \le 2\vxi^T(\hmu - \vu)$ and thus
    \begin{equation}
        \hat R - R
        = 
        \frac{\hat R^2 - R^2}{\hat R + R} 
        \le 
        \frac{2\vxi^T(\hmu - \vu)}{\hat R + R}
        = 2\vxi^T\vtheta
        \le 2 Y_\vu.
    \end{equation}
    We have proved \eqref{eq:oi-unimodal-Y-u}.
    Applying \eqref{eq:concentration-Y-u} completes the proof 
of the Theorem.
\end{proof}
The oracle inequality of \Cref{thm:unimodal-soi} is sharp (that is, it has leading constant 1),
but it is an oracle inequality with respect to the loss $\scalednorm{\cdot}$ rather than to the squared loss $\scalednorms{\cdot}$.
An oracle inequality with respect to the loss $\scalednorms{\cdot}$ is stronger than an oracle inequality
with respect to the loss $\scalednorm{\cdot}$.
Indeed, if an estimator $\hmu$ satisfies $\scalednorms{\hmu - \vmu} \le \min_{\vu\in E}\scalednorms{\vu - \vmu} + \varepsilon$
for some set $E$ and some $\varepsilon>0$, then the inequality $\sqrt{a+b}\le \sqrt a + \sqrt b$ for all $a,b\ge 0$ yields
$\scalednorm{\hmu - \vmu} \le \min_{\vu\in E}\scalednorm{\vu - \vmu} + \varepsilon^{1/2}$.

An oracle inequality similar to that of \Cref{thm:unimodal-soi} can be
obtained for the nonparametric rate $n^{-2/3}$.
\begin{thm}
    \label{thm:unimodal-soi-n23}
    There exists an absolute constant $c>0$ such that the following holds.
    Let $\vmu\in\R^n$.
    If $\vxi\sim\mathcal N(\vzero,\sigma^2 I_{n\times n})$
    then for all $x>0$, we have
    \begin{equation}
        \scalednorm{\ls(\mathcal U) - \vmu}
        \le
        \min_{\vu\in\mathcal U}
        \left[
            \scalednorm{\vu - \vmu}
            +
            2 c \sigma \left(\frac{\sigma +V(\vu)}{\sigma n}\right)^{1/3}
        \right]
        + \frac{2 (2+\sqrt 2) \sigma \sqrt{x+\log(en)}}{\sqrt n}
        \label{eq:oi-unimodal-n23}
    \end{equation}
    with probability at least $1-e^{-x}$,
    where $V(\cdot)$ is defined in \eqref{eq:def-V}.
\end{thm}
\begin{proof}
    Let $\hmu=\ls(\mathcal U)$ and let $\vu$ be a minimizer of the right hand side of the oracle inequality.
    Define
    \begin{equation}
        t\coloneqq c \sigma \left(1+\frac{V(\vu)}{\sigma}\right)^{1/3} n^{1/6},
        \label{eq:def-t-n16}
    \end{equation}
    where $c>0$ is the absolute constant from \eqref{eq:chaining-chatterjee}.
    Define the random variables $Z_1,...,Z_m$ by
    \begin{equation}
        Z_m \coloneqq
        \sup_{\vv\in K_m: \euclidnorm{\vv-\vu}>t}
        \frac{2\vxi^T(\vv - \vu)}{\euclidnorm{\vv- \vu}},
        \qquad
        m=1,...,n.
        \label{eq:def-Z_m}
    \end{equation}
    We will prove that the oracle inequality of the Theorem holds on the event
    \begin{equation}
        \max_{m=1,...,n}Z_m \le t + 2\sigma \sqrt{\log(en)} + \sqrt{2(x+\log n)}.
        \label{eq:event-Z_m}
    \end{equation}
    Let $\hat R=\euclidnorm{\hmu - \vmu}$
    and let $R = \euclidnorm{\vu - \vmu}$.
    If $\euclidnorm{\hmu - \vu} \le t$ then by the triangle inequality
    $\hat R - R \le \euclidnorm{\vu-\hmu} \le t$
    and the oracle inequality of the Theorem holds.
    On the other hand, if $\euclidnorm{\hmu - \vu} > t$ then
    using \eqref{eq:ls-nonconvex} and the triangle inequality  we have
    \begin{equation}
        \hat R - R =
        \frac{\hat R^2 - R^2}{\hat R + R}
        \le
        \frac{2\vxi^T(\hmu - \vu)}{\hat R + R}
        \le
        \frac{2\vxi^T(\hmu - \vu)}{\euclidnorm{\hmu - \vu}}
        \le
        \max_{m=1,...,n}
        2 Z_m.
    \end{equation}
    On the event \eqref{eq:event-Z_m}, the right hand side of the previous display is bounded
    from above by $2t + 2(2+\sqrt 2)\sqrt{x+\log(en)}$
    and the oracle inequality of the Theorem holds.

    It remains to bound the probability of the event \eqref{eq:event-Z_m}.
    By the concentration inequality for suprema of Gaussian processes and the union bound over $m=1,...,n$, we have
    with probability at least $1-e^{-x}$,
    for all $m=1,...,n$,
    \begin{equation}
        Z_m\le E[Z_m] + \sigma \sqrt{2(x+ \log n)},
        \qquad
        \max_{m=1,...,n}
        Z_m
        \le
        \max_{m=1,...,n}
        \E[Z_m]
        + \sigma\sqrt{2(x+\log n)}.
    \end{equation}
    It remains to bound $\E[Z_m]$ from above.
    \begin{lemma}
        \label{lemma:E-Z_m}
        For all $m=1,...,n$, the variable $Z_m$
        defined in \eqref{eq:def-Z_m} satisfies 
        \begin{equation}
            \label{eq:E-Z_m}
            \E[Z_m]
            \le t + 2 \sqrt{\log(en)}.
        \end{equation}
    \end{lemma}
    The proof of \Cref{lemma:E-Z_m} is given in \Cref{s:lemma-T_m}.
    This proves that the event \eqref{eq:event-Z_m}
    has probability at least $1-e^{-x}$.
\end{proof}

\subsection{
    Proof of \Cref{lemma:Y_vu},
    \Cref{lemma:T_m}
    and \Cref{lemma:E-Z_m}
}
\label{s:lemma-T_m}
\begin{proof}[Proof of \Cref{lemma:Y_vu}]
    Using \eqref{eq:equivalence-delta} we have
    $Y_\vu= \max_{m=1,...,n} \euclidnorm{\Pi_{\mathcal T_{m,\vu}}(\vxi)}$.
    We apply
    \eqref{eq:concentration-delta-not-squared}
    to $L= \mathcal T_{m,\vu}$ for all $m=1,...,n$.
    The union bound completes the proof.
\end{proof}

\begin{proof}[Proof of \Cref{lemma:T_m}]
    Let $m=1,...,n$,
    $k=k(\vu)$ and let
    $(T_1,...,T_{k})$ be a partition of $\{1,...,n\}$ such that
    $\vu$ is constant on each $T_l$ and $T_l$ is convex for all $l=1,...,k$.
    Let $l^*\in\{1,...,k\}$ be the unique integer such that $m\in T_{l^*}$,
    and let $T^* = T_{l^*}$.
    Let $\vv\in K_m$.
    Then for all $l <l^*$, the sequence $(\vv - \vu)_{T_l}$ is non-increasing
    and for all $l> l^*$, the sequence $(\vv - \vu)_{T_l}$ is nondecreasing. 
    Furthermore, if $A = T^* \cap \{1,...,m\}$ and $B = T^* \cap \{m+1,...,n\}$,
    the sequence $(\vv - \vu)_{A}$ is non-increasing and
    the sequence $(\vv - \vu)_{B}$ is nondecreasing.
    We have proved the inclusion
    \begin{equation}
        \mathcal T_{m,\vu} \subset 
        \mathcal C \coloneqq
        \mathcal S^\downarrow_{|T_1|}
        \times
        ...
        \times
        \mathcal S^\downarrow_{|T_{l^* -1}|}
        \times 
        \mathcal S^\downarrow_{|A|}
        \times 
        \mathcal S^\uparrow_{|B|}
        \times 
        \mathcal S^\uparrow_{|T_{l^*+1}|}
        \times
        ...
        \times
        \mathcal S^\uparrow_{|T_{k}|}.
    \end{equation}
    where for all integer $q\ge 1$, $\mathcal S^\uparrow_q$ is the cone of nondecreasing sequences in $\R^q$
    and $\mathcal S^\downarrow_q$ is the cone of non-increasing sequences in $\R^q$.
    Using \eqref{eq:delta-inclusion}, \eqref{eq:delta-product}
    and \eqref{eq:delta-increasings-logn} we obtain
    \begin{equation}
        \delta(\mathcal T_{m,\vu})
        \le
        \delta(\mathcal C)
        \le
        \log(e|A|)
        + \log(e|B|)
        +
        \sum_{l=1,...,k:l\ne l^*}
        \log( e|T_l|).
    \end{equation}
    Using Jensen's inequality with the fact that $|A| + |B| + \sum_{l=1,...,k:l\ne l^*} |T_l| = n$,
    we obtain
    $
        \delta(\mathcal T_{m,\vu})
        \le
        (k + 1) \log \frac{en}{k + 1}
    $.
\end{proof}

\begin{proof}[Proof of \Cref{lemma:E-Z_m}]
    Let $m=1,...,n$ be fixed.
    As $\vu\in\mathcal U$, there exists $m_\vu\in\{1,...,n\}$ such that $\vu \in K_{m_\vu}$.
    Define $k\coloneqq\min(m, m_\vu)$ and $j\coloneqq\max(m,m_\vu)$.
    Let $T \coloneqq \{1,...,k\}$, $E \coloneqq \{k+1,...,j-1\}$ and $ S \coloneqq \{j,...,n\}$.
    Then for all $\vv\in K_m$, by definition of $T,E$ and $S$ we have
    \begin{equation}
        \vv_T \in \mathcal S^\downarrow_{|T|},
        \vu_T \in \mathcal S^\downarrow_{|T|},
        \quad
        \vv_S \in \mathcal S^\uparrow_{|S|},
        \vu_S \in \mathcal S^\uparrow_{|S|},
        \quad
        (\vv - \vu)_E \in \mathcal S^\uparrow_{|E|}\cup\mathcal S^\downarrow_{|E|}.
    \end{equation}
    Thus, the quantity $\E[Z_m]$ is bounded from above by
    \begin{align}
        &\quad\E\sup_{\vv\in K_m: \euclidnorm{\vv-\vu}>t}
        (1/\euclidnorm{\vv- \vu}) \;\; \vxi_T^T(\vv - \vu)_T 
        \label{eq:Z_m:other}
        \\
        +
        &\quad\E\sup_{\vv\in K_m: \euclidnorm{\vv-\vu}>t}
        (1/\euclidnorm{\vv- \vu}) \;\; \vxi_S^T(\vv - \vu)_S 
        \label{eq:Z_m:other2}
        \\
        +
        &\quad\E\max_{L\in\{ \mathcal S^\uparrow_{|E|}, \mathcal S^\downarrow_{|E|}\}}\sup_{\vx\in L:\euclidnorm{\vx}\le 1} \vxi_E^T\vx.
        \label{eq:Z_m:delta}
    \end{align}
    By definition of the statistical dimension and using the fact that the maximum of two positive numbers
    is bounded from above by their sum, we have that \eqref{eq:Z_m:delta} is
    bounded from above
    by $\delta(\mathcal S^\uparrow_{|E|})^{1/2} + \delta(\mathcal S^\downarrow_{|E|})^{1/2}\le 2\log(e|E|)^{1/2}\le 2 \log(en)^{1/2}$.
    We now bound \eqref{eq:Z_m:other} from above, while \eqref{eq:Z_m:other2} can be bounded similarly.
    For all $\vv\in K_m$ such that $\euclidnorm{\vv-\vu} > t$,
    if
    $\alpha=t/\euclidnorm{\vv -\vu}$
    we have
    \begin{equation}
        (1/\euclidnorm{\vv- \vu})\vxi_T^T(\vv - \vu)_T
        = (\alpha/t) \vxi_T^T(\vv - \vu)_T
        = (1/t) \vxi_T^T(\alpha \vv + (1-\alpha)\vu - \vu)_T
    \end{equation}
    and $\alpha\in[0,1]$.
    By convexity we have
    $\vtheta = \alpha \vv_T + (1-\alpha)\vu_T\in\mathcal S^\downarrow_{|T|}$
    and 
    \begin{equation}
        \euclidnorm{\vtheta - \vu_T} = \alpha \euclidnorm{\vv_T - \vu_T} \le \alpha \euclidnorm{\vv - \vu} =  t.
    \end{equation}
    Thus, \eqref{eq:Z_m:other} is bounded from above by
    \begin{equation}
        \frac 1 t \E \sup_{\vtheta \in \mathcal S^\downarrow_{|T|}: \euclidnorm{\vtheta - \vu_T} \le t}
        \vxi_T^T(\vtheta - \vu_T).
    \end{equation}
    By \eqref{eq:chaining-chatterjee}, the previous display is bounded from above by $t/2$.
    Finally, we have
    \begin{equation}
        \E[Z_m]
        \le \eqref{eq:Z_m:other} + \eqref{eq:Z_m:other2} + \eqref{eq:Z_m:delta}
        \le t/2 + t/2 + 2 \sqrt{\log(en)}.
    \end{equation}
\end{proof}

\subsection{
    Proofs of \Cref{thm:convex-unimodal,thm:convex-rate23}
}
\label{s:proofs-convex-unimodal}

\begin{proof}[Proof of \Cref{thm:convex-unimodal}]
    Let $\vu\in\mathcal U$ be a unimodal sequence.
    Define the random variable
    \begin{equation}
        G \coloneqq \vxi^T(\vu - \vmu) / \euclidnorm{\vu - \vmu}
        \quad
        \text{ if } \vu\ne\vmu
        \text{ and } G = 0 \text{ otherwise}.
        \label{eq:def-G}
    \end{equation}
    Let $Y_\vu$ be defined in \eqref{eq:def-Y-u}.
    Let $\tilde R \coloneqq \euclidnorm{\hmu - \vmu}$
    and $R \coloneqq \euclidnorm{\vu - \vmu}$.
    We first prove that almost surely,
    \begin{equation}
        \scalednorm{\hmu - \vmu}
        \le
        \max\left(
            \scalednorm{\vu - \vmu},
            \frac{2Y_\vu + G}{\sqrt n}
        \right)
        \label{eq:claim:Y_u}
    \end{equation}
    It is enough to prove that
    $\tilde R>R$ implies $\tilde R\le 2Y_\vu+G$.
    Assume that $\tilde R>R$.
    Inequality \eqref{eq:kkt-noise} yields
    $
    \tilde R^2
    \le
    \vxi^T(\hmu - \vmu)
    $ and thus
    \begin{equation}
    \tilde R^2
    \le
    \vxi^T(\hmu - \vmu)
    =
    \frac{\vxi^T(\hmu - \vu)}{\euclidnorm{\hmu - \vu}} \euclidnorm{\hmu - \vu} + G \euclidnorm{\vu - \vmu}
    \le Y_\vu \euclidnorm{\hmu - \vu} + G R.
    \end{equation}
    where we used that $\hmu\in\mathcal U$ since a convex
    sequence is unimodal.
    We have $R < \tilde R$ and
    by the triangle inequality, 
    $\euclidnorm{\hmu - \vu} \le \tilde R + R 
    < 2 \tilde R$,
    which proves that $\tilde R^2 \le 2 Y_\vu \tilde R + G \tilde R$. 
    Dividing by $\tilde R$ completes the proof of \eqref{eq:claim:Y_u}.

    We now prove the oracle inequality.
    Since $G$ is centered Gaussian
    with variance at most $\sigma^2$ we have $\mathbb P(G>\sigma\sqrt{2(x+\log n)})\le e^{-x}/n\le e^{-x}$.
    The concentration inequality \eqref{eq:concentration-Y-u} and the union bound completes the proof.
\end{proof}
\begin{proof}[Proof of \Cref{thm:convex-rate23}]
    Let $\vu$ be a minimizer of the right hand side.
    Let $\tilde R \coloneqq \euclidnorm{\hmu - \vmu}$
    and $R \coloneqq \euclidnorm{\vu - \vmu}$.
    Define $t>0$ by \eqref{eq:def-t-n16}
    and define the random variables $Z_1,...,Z_n$
    by \eqref{eq:def-Z_m}.
    We first prove that almost surely,
    \begin{equation}
        \tilde R
        \le
        \max\left(
            R + t,
            \;\;
            2 \max_{m=1,...,n}Z_m + G
        \right).
        \label{eq:claim-convex-Z_m}
    \end{equation}
    It is enough to prove that $\tilde R > R + t$ implies
    $\tilde R \le 2 \max_{m=1,...,n} Z_m + G$.
    If $\tilde R > R + t$ then
    by the triangle inequality,
    $
    \euclidnorm{\hmu - \vu}\ge
        \tilde R - R
        > t + R - R
        = t
    $, so that
    by \eqref{eq:kkt-noise} we have
    \begin{equation}
        \tilde R^2 \le 
        \vxi^T(\hmu - \vmu)
        \frac{\vxi^T(\hmu - \vu)}{\euclidnorm{\hmu - \vu}} \euclidnorm{\hmu - \vu} + \frac{\vxi^T(\vu - \vmu)}{\euclidnorm{\vu - \vmu}} R
        \le
        \left(\max_{m=1,...,n}Z_m\right) \euclidnorm{\hmu - \vu}
        + G R,
    \end{equation}
    where $G$ is the random variable \eqref{eq:def-G}.
    If $\tilde R > R+t$ we have
    $R\le \tilde R$
    and
    $\euclidnorm{\hmu - \vu}\le \tilde R + R \le 2\tilde R$,
    so that dividing both sides of the previous display by $\tilde R$
    yields
    that
    $\tilde R \le 2 \max_{m=1,...,n}Z_m + G$.
    The proof of \eqref{eq:claim-convex-Z_m} is complete.

    We now prove the probability estimate.
    Since $G$ is centered Gaussian
    with variance at most $\sigma^2$ we have $\mathbb P(G>\sigma\sqrt{2(x+\log(en) )})\le e^{-x}/(en)\le e^{-x}$.
    In the proof of
    \Cref{thm:unimodal-soi-n23},
    we showed that the event \eqref{eq:event-Z_m}
    has probability at least $1-e^{-x}$.
    The union bound completes the proof.
\end{proof}

\section{Proof of \Cref{cor:rate23} and \Cref{cor:rate45}}
\label{s:proof-iso}

The following result is established in \citet{chatterjee2014new}.
\begin{lemma}[\cite{chatterjee2014new}]
    There exists an absolute constant $c>0$ such that the following holds.
    Let $\vu \in \increasings$ and $\vxi\sim\mathcal N(\vzero,\sigma^2 I_{n\times n})$.
    Then
    \begin{equation}
        \E \sup_{\vtheta \in\increasings:\euclidnorm{\vtheta - \vu}\le 1}
        \vxi^T(\vtheta - \vu)
        \le \frac{t^2}{2}
        \quad
        \text{ for all }
        t\ge c \sigma \left(1+\frac{V(\vu)}{\sigma}\right)^{1/3} n^{1/6}
        .
        \label{eq:chaining-chatterjee}
    \end{equation}
\end{lemma}

\begin{proof}[Proof of \Cref{cor:rate23}]
    Combining \eqref{eq:chaining-chatterjee}
    and \Cref{thm:isomorphic} completes the proof.
\end{proof}
\begin{proof}[Proof of \Cref{cor:rate45}]
    Let $\vu$ be a minimizer of the right hand side
    of the oracle inequality of \Cref{cor:rate45}.
    By rescaling we may assume that $\sigma=1$.
    Let $R=R_{\vu}$.
    As in the proof of \Cref{cor:rate23},
    we apply \Cref{thm:isomorphic}.
    Let $r>0$.
    Let $S(\vu,r) = \{\vv\in\convexs, \scalednorm{\vv-\vu} \le r\}$.
    By Dudley entropy bound (cf. \cite[Corollary 13.2]{boucheron2013concentration}),
    we obtain
    \begin{align}
    \E \sup_{\vv\in S(\vu,r)} 
    \frac{
        \vxi^T (\vv-\vu)
    }{\sqrt n}
    & \le 12 \int_0^r \sqrt{\log M(\epsilon, S(\vu,r), \scalednorm{\cdot})} d\epsilon, \\
    & \le \bar\kappa \log(en)^{5/8} r^{3/4}(r^2 + R^2)^{1/8},
    \end{align}
    where $\bar\kappa>0$ is an absolute constant, 
    $M(\epsilon, S(\vu,r), \scalednorm{\cdot})$ is the $\epsilon$-entropy
    of $S(\vu,r)$ in the $\scalednorm{\cdot}$ norm,
    and the second inequality is proved in \cite[(25)]{guntuboyina2013global}.
    The constant $1/\sqrt{n}$ on the left hand side
    due to the fact that the Gaussian process is normalized with respect to the metric $\scalednorm{\cdot}$, i.e.,
    for all vectors $\vv,\vv'$,
    $\scalednorms{\vv-\vv'} =\E[ ( \vxi^T(\vv-\vv') / \sqrt{n})^2]$.
    Let now $t = r\sqrt{n}$.
    After rearranging, the previous inequality becomes
    \begin{equation}
    \E \sup_{\vv\in \convexs:\; \euclidnorms{\vv-\vu}\le t} 
        \vxi^T (\vv-\vu)
    \le
        \bar\kappa \log(en)^{5/8} n^{1/8} t^{3/4}\left(\frac{t^2}{n} + R^2\right)^{1/8}.
    \label{eq:RHS-rate45-tobound}
    \end{equation}
    Let 
    \begin{equation}
        t_* = \left(2 \bar\kappa 2^{1/8}\right)^{4/5}\sqrt{\log(en)}R^{1/5} n^{1/10},
    \end{equation}
    and choose the absolute constant
    $\kappa \coloneqq 4 \bar\kappa^2 2^{1/4}$.
    With this choice of $\kappa$ and $t_*$,
    \eqref{eq:rate45condition} is equivalent to $t_*^2 / n \le R^2$.
    Thus, for $t=t_*$, the right hand side of 
    \eqref{eq:RHS-rate45-tobound}  does not exceed
    \begin{equation}
        \bar\kappa2^{1/8} \log(en)^{5/8} n^{1/8} t_*^{3/4} \le t_*^2 /2.
    \end{equation}
    Applying \Cref{thm:isomorphic} completes 
    the proof of \Cref{cor:rate45}.
\end{proof}

\smallskip

{\bf Acknowledgement.}
The author thanks Alexandre Tsybakov for helpful comments
and Philippe Rigollet
for informing him of the unimodal regression result 
of \citet{flammation2016seriation}.
This work was supported by GENES and by the French National Research Agency (ANR) under the grants 
IPANEMA (ANR-13-BSH1-0004-02), and Labex ECODEC (ANR - 11-LABEX-0047).

\bibliographystyle{plainnat}
\bibliography{db}

\end{document}

\section{Other cones}

\begin{example}[$\beta$-constrained sequences]
\label{example:beta}
For any positive integer $\beta<n$, consider the weights vector
$\vomega^{[\beta]} = (\omega^{[\beta]}_0,...,\omega^{[\beta]}_\beta)^T\in\mathbf{R}^{\beta+1}$ defined by
\begin{equation}
    \omega^{[\beta]}_k = (-1)^{\beta - k} \binom{\beta}{k},
    \qquad
    k=0,...,\beta.
    \label{eq:vomega-beta}
\end{equation}
Define the cone 
\begin{align}
    \conebeta
    &\coloneqq
    \{ \vu=(u_1,\dots,u_n)^T\in\Rn: \; (\vomega^{[\beta]})^T (u_i,u_{i+1},...,u_{i+\beta})^T\ge 0, \  i=1,\dots,n-\beta \}, \\
    &\coloneqq
    \{ \vu=(u_1,\dots,u_n)^T\in\Rn: \; D_{n-\beta+1}D_{n-\beta + 2}...D_n\vu \ge \vzero = (0,...,0)^T \},
\end{align}
where $D_{n-\beta+1},...,D_n$ are the matrices defined in
\eqref{eq:def-Dmatrix}.
In particular, we have
$\vomega^{[1]} = (-1,1)^T$
and
$\mathcal{S}^{[1]} = \increasings$ is the cone of
nondecreasing sequences,
$\vomega^{[2]} = (1,-2,1)^T$
and $\mathcal{S}^{[2]} = \convexs$ is the cone
of convex sequences.
The notation $\beta$ has been chosen to highlight 
the similarity between the cones
$\conebeta$
and $\beta$-smoothness classes
in nonparametric statistics.
In univariate regression for instance,
the minimax rate of estimation under the loss \eqref{eq:loss}
for smoothness classes such as Hölder or Sobolev balls
is proportional to $n^{-2\beta/(2\beta+1)}$,
where $\beta$ is the smoothness of the class.
Inequalities \eqref{eq:rate23-example-nonoi}
and \eqref{eq:rate45-example-nonoi} show 
that the rates of convergence of the Least Squares estimator
over the cones $\conebeta$
under the loss \eqref{eq:loss}
are $n^{-2\beta/(2\beta+1)}$ up to logarithmic factors for $\beta=1,2$.

If $\vu=(u_1,...,u_n)^T\in\conebeta$,
we say that $\vu$ is a piecewise polynomial function of degree $d$ with $k$ pieces
if there exist
polynomials $Q_1,...,Q_k$ of degree at most $d$ 
and
a partition $(T_1,...,T_k)$ of $\{1,...,n\}$
such that
\begin{equation}
    u_i = Q_j(i),
    \qquad
    i\in T_j,
    \qquad
    j=1,...,k.
    \label{eq:polynomial}
\end{equation}
If $\vu\in\conebeta$, define $s_\beta(\vu)\ge 1$
as the smallest integer $s$ such that
$\vu$ is a piecewise polynomial function of degree $\beta-1$ with $s$ pieces.
Note that $s_1(\cdot) = k(\cdot)$ for nondecreasing sequences,
and $s_2(\cdot) = q(\cdot)$ for convex sequences.
The cone $\conebeta$ is endowed with with
the lineality space
\begin{equation}
    \lineality{\conebeta} 
    = \{\vu\in\conebeta: s_\beta(\vu) = 1 \},
\end{equation}
which is
the subspace of polynomials of degree at most $\beta - 1$.
In \Cref{s:delta} we will derive upper bounds on
the statistical dimension of the cones $\conebeta$
for all $\beta\ge 2$.
These bounds lead to sharp oracle inequalities for
$\ls(\conebeta)$
similar 
to \eqref{eq:increasings-adapt-example-oi}
and \eqref{eq:convexs-adapt-example-oi}.
\end{example}

\begin{example}[$m$-monotone sequences]
    \label{example:monotone}
    Define the linear 
    operator $\nabla:\Rn\rightarrow\Rn$
    by
    \begin{equation}
        \nabla \vu
        = (u_2-u_1,...,u_n-u_{n-1}, -u_n),
        \qquad
        \vu=(u_1,...,u_n)^T \in\Rn.
    \end{equation}
    If $m\ge0$ is an integer,
    the cone of $m$-monotone sequences is defined as
    \begin{equation}
        \mathcal M^m_n = \{\vu\in\Rn: \nabla^m \vu \ge 0 \}.
    \end{equation}
    For density estimation,
    $m$-monotone functions have been
    studied in \cite{balabdaoui2007estimation,balabdaoui2010estimation}.
    Simple algebra shows that $\mathcal M^0_n = \R^{n+}$
    and $\mathcal M^1_n=\increasings\cap ( - \R^{n+})$,
    where $\R^{n+}$ is the nonnegative orthant.
    For $m=2$,    $\mathcal M^2_n=\convexs \cap ( - \increasings) \cap \R^{n+}$
    is the cone of convex, non-increasing and nonnegative sequences.
    For all $m\ge 1$, we have
    \begin{equation}
        \mathcal M^m_n = \cap_{l=0}^m \left( (-1)^{m-l} \conebetaat{l}\right),
        \label{eq:monotone-cap}
    \end{equation}
    with the convention $\conebetaat{0} = \R^{n+}$.
    This implies that $\mathcal M^m_n \subset(-1)^{m-1} \increasings$.
    Using the monotonicity of the statistical dimension \eqref{eq:delta-inclusion}, 
    we obtain $\delta(\mathcal M^m_n) \le \delta((-1)^{m-1}\increasings) = \delta(\increasings) \le \log(en)$.
    This simple monotonicity argument cannot
    be used to bound from above the statistical dimension of the cones
    $\convexs$ or $\conebeta$ defined above.
    For all $m\ge 0$, the cone $\mathcal M^m_n$ contains no linear subspace and
    \begin{equation}
        \lineality{\mathcal M^m_n} = \{\vzero \}.
    \end{equation}
    This is opposed to the cones $\conebeta$,
    since the lineality space of $\conebeta$ has dimension $\beta$.
\end{example}

\begin{example}[Cone with arbitrary weights vector $\vomega$]
\label{example:weights}
Given a positive integer $m$ with $m\le n$
and a vector of weights $\vomega\in\Rm$,
define the cone $\K_n^\vomega$ by 
\begin{multline}
    \K^\vomega_n =
    \Big\{
            \vu=(u_1,\dots,u_n)^T\in\Rn: \; \vomega^T(u_i,u_{i+1},...,u_{i+m-1})^T\ge 0 \\ 
            \quad \text{ for all }  i=1,\dots,n-m+1
    \Big\}.
\end{multline}
The cone $\K^\vomega_n$ is a closed and convex subset of $\Rn$.
The sets $\increasings$, $\convexs$
and $\conebeta$ defined above
are examples of cones of this form.
\end{example}

\begin{example}[Polyhedral cones]
\label{example:polyhedral}
All the examples above are particular cases of convex polyhedral cones.
Given a matrix $A$ with $r$ rows and $n$ columns, define 
the cone
\begin{equation}
    \C_A =
    \{
            \vu\in\Rn,
            \quad
            A \vu
            \le \vzero = (0,...,0)^T
    \},
\end{equation}
where $\le$ denotes the component wise comparison in $\mathbf{R}^r$.
The polyhedral cones in $\Rn$ are the sets of
the form $\C_A$ where $A$ is a matrix with $n$ columns.
The lineality space of $\C_A$ is the kernel of matrix $A$.
The cone of nondecreasing sequences
and the cone of convex sequences are polyhedral cones
with $\increasings = \C_{-D_n}$ and $\convexs = \C_{-D_{n-1}D_n}$.
For cones of higher order defined in \Cref{example:beta}, we have
$\conebeta = \C_{A_\beta}$ where 
\begin{equation}
    A_\beta=-D_{n-\beta+1}D_{n-\beta+2}...D_n.
\end{equation}
Given a vector $\vomega\in\Rm$ as in
\Cref{example:weights},
the cone
$\K^\vomega_n$ satisfies
$\K^\vomega_n=\C_{A_\vomega}$ where
$A_\vomega = (a_{ij})_{i=1,...,n;j=1,...,n-m+1}$
is the matrix
\begin{equation}
    a_{ij}= - \omega_{j+i-1} \quad\text{ if } 1\le i+j-1\le m \qquad \text{and} \qquad a_{ij}= 0 \quad \text{ otherwise}.
\end{equation}
\end{example}


\section{Results for other cones}

\subsection{Minimax regret bounds for $\conebeta$}

The argument behind \Cref{thm:delta-convex}
can be used 
to recursively control the statistical dimensions
of the cones $\conebeta$ for $\beta\ge 3$.

\begin{thm}
    \label{thm:delta-beta}
    Let $\beta,n$ be integers such that $1 \le \beta < n$.
    Then
    \begin{equation}
        \label{eq:bound-delta-beta}
        \delta(\conebeta)
        =
        \E \euclidnorms{\Pi_{\conebeta}(\vg)} \le C(\beta) \log(en),
    \end{equation}
    where
    $\vg\sim\mathcal N (\vzero,I_{n\times n})$
    and
    $C(\beta) = 3\cdot4^{\beta-1} - 2$.
\end{thm}

\subsubsection{Proofs: Upper bounds on statistical dimensions of cones}
\label{s:delta}

For any $k=1,...,n-1$,
let $S_k=\{1,...,k\}$ and $T_k=\{k+1,...,n\}$.
For any subset $T\subset\{1,...,n\}$ and any vector $\vtheta\in\Rn$,
denote by $\vtheta_T\in\mathbf{R}^T$ the restriction of $\vtheta$ to $T$.
For any $\vg,\vtheta\in\Rn$ and any $k=1,...,n-1$,
\begin{equation}
    \vg^T\vtheta = 
    \vg_{T_k}^T\vtheta_{T_k}
    +
    \vg_{S_k}^T\vtheta_{S_k}
    .
\end{equation}
For any closed convex cone $\K$, denote  by $\Pi_{\K}$ the projection onto $\K$.

\begin{lemma}
    \label{lemma:delta}
    Let
    $\K\subset\Rn$ be a closed convex cone.
    Assume that there exists a collection
    $\{(\coneleft_k,\coneright_{n-k}), k=1,...,n-1\}$ where
    $\coneleft_k\subset \R^k,\coneright_{n-k}\subset \R^{n-k}$
    are closed convex cones such that the following holds.
    For all $\vtheta\in\K$,
    there exists $k\in\{1,...,n-1\}$ such that
    \begin{equation}
        \vtheta_{S_k}\in\coneleft_k,
        \qquad
        \vtheta_{T_k}\in\coneright_{n-k}.
        \label{eq:vtheta-Sk-Tk}
    \end{equation}
    Then
    \begin{equation}
        \delta(\K)
        \le 2 d^* + 6\log(n-1),
    \qquad
    \text{where}
    \quad
    d^* \coloneqq \max_{k=1,...,n-1}\left[
                \delta(\coneleft_k) + 
                \delta(\coneright_{n-k})
        \right].
    \end{equation}
\end{lemma}

\begin{proof}
    Let $\vg\sim\mathcal N (\vzero,I_{n\times n})$.
    If $\vtheta\in\K$ is such that 
    \eqref{eq:vtheta-Sk-Tk}
    holds for some $k=1,...,n-1$,
    then
    by the Cauchy-Schwarz inequality,
    \begin{equation}
        \vtheta^T \vg
        =
        \vtheta^T_{S_k} \vg_{S_k}
        +
        \vtheta^T_{T_k} \vg_{T_k}
        \le
        \sqrt{
            \euclidnorms{\vtheta^T_{S_k}}
            +
            \euclidnorms{\vtheta^T_{T_k}}
        }
        \sqrt{
            Z^{L}_k
            +
            Z^{R}_{n-k}
        }
        =
        \euclidnorm{\vtheta}
        \sqrt{
            Z^{L}_k
            +
            Z^{R}_{n-k}
        },
        \label{eq:take-supr-theta}
    \end{equation}
    where we used 
    the notation
    \begin{equation}
            Z^{L}_k
            \coloneqq
            \left(
                \sup_{\vu\in\coneleft_k: \euclidnorms{\vu} \le 1}
                \vg_{S_k}^T\vu
            \right)^2
            ,
            \qquad
            Z^{R}_{n-k}
            \coloneqq
            \left(
                \sup_{\vu\in\coneright_{n-k}: \euclidnorms{\vu} \le 1}
                \vg_{T_k}^T\vu
            \right)^2
            .
    \end{equation}
    By 
    \eqref{eq:equivalence-delta},
    we have almost surely
    \begin{equation}
            Z^{L}_k
            =
            \euclidnorms{\Pi_{\coneleft_k}(\vg_{S_k})}
            ,
            \qquad
            Z^{R}_{n-k}
            =
            \euclidnorms{\Pi_{\coneright_{n-k}}(\vg_{T_k})}
            .
    \end{equation}
    Similarly, let $Z \coloneqq \euclidnorms{\Pi_{\K}(\vg)} =
    (\sup_{\vtheta\in\K: \euclidnorms{\vtheta}\le 1}
        \vtheta^T \vg
    )^2$.
    Using 
    \eqref{eq:vtheta-Sk-Tk}
    and by taking the supremum over $\euclidnorms{\vtheta}\le 1$
    in
        \eqref{eq:take-supr-theta}
    , we have established
    \begin{equation}
        \sqrt Z 
        =
        \sup_{\vtheta\in\K: \euclidnorms{\vtheta} \le 1}
        \vtheta^T \vg
        \le
        \max_{k=1,...,n-1}
        \sqrt{
            Z^{L}_k
            +
            Z^{R}_{n-k}
        }.
    \end{equation}
    For any fixed $k$,
    $\vg_{S_k}$ and $\vg_{T_k}$ are independent,
    thus
    $Z^{L}_k$
    and $Z^{R}_{n-k}$ are independent.
    By independence, for all $\lambda>0$, 
    \begin{equation}
        \E e^{\lambda Z}
        \le
        \sum_{k=1}^{n-1}
        \E \exp(
        \lambda Z^{L}_k 
        + \lambda  Z^{R}_{n-k}
        )
        =
        \sum_{k=1}^{n-1}
        \E[ \exp(
        \lambda Z^{L}_k 
        )]
        \;
        \E[ \exp(
        \lambda  Z^{R}_{n-k}
        )].
        \label{eq:sumk-independence}
    \end{equation}
    Applying the moment generating function bound given in \cite[Sublemma A.3]{amelunxen2014living}, we obtain that for any $\lambda \in (0,1/4)$,
    \begin{equation}
        \E[ \exp(
        \lambda Z^{L}_k 
        )]
        \;
        \E[ \exp(
        \lambda  Z^{R}_{n-k}
        )]
        \le
        \exp\left[
            \left(\frac{2\lambda^2}{1-4\lambda} + \lambda\right)
            \left(
                \delta\left(\coneleft_k\right) + 
                \delta\left(\coneright_{n-k}\right)
            \right)
        \right].
    \end{equation}
    Let $\lambda^* = 1/6$. The previous
    display with $\lambda=\lambda^*$ and the definition of $d^*$ yield
    \begin{equation}
        \E[ \exp(
        \lambda^* Z^{L}_k 
        )]
        \;
        \E[ \exp(
        \lambda^*  Z^{R}_{n-k}
        )]
        \le
        \exp\left[
                2\lambda^* d^*
        \right],
    \end{equation}
    for all $k=1,...,n-1$.
    Plugging these bounds back into
    \eqref{eq:sumk-independence}, we obtain 
    \begin{equation}
        \E e^{\lambda^* Z}
        \le
        (n-1) 
        e^{
            2\lambda^* d^*
        }=
        e^{\lambda^*[2d^* + 6 \log(n-1)]}.
    \end{equation}
    The function $t\rightarrow \exp(\lambda^* t)$ is convex,
    so applying Jensen's inequality yields
    $\delta(\K) = \E[Z] \le 2d^* + 6\log(n-1)$.
\end{proof}

\begin{proof}[Proof of \Cref{thm:delta-convex}]
    Define the convex cones
    \begin{equation}
        \coneleft_k = \mathcal{S}^\downarrow_k,
      \qquad
      \coneright_{n-k} = \mathcal{S}^\uparrow_{n-k},
      \qquad
      k=1,...,n-1,
      \label{eq:cones-convex}
    \end{equation}
    i.e., the cone of non-increasing sequences in $\R^k$
    and the cone of nondecreasing sequences in $\R^{n-k}$.
    A convex sequence $\vtheta\in\K^C_{x_1,...,x_n}$
    must be first non-increasing and then nondecreasing,
    so 
    \eqref{eq:vtheta-Sk-Tk}
    holds for some $k=1,...,n-1$.

    We apply \Cref{lemma:delta}
    with $\K=\K^C_{x_1,...,x_n}$ and the cones defined in \eqref{eq:cones-convex}. 
    By \eqref{eq:delta-sum1overk}, $d^*\le 2\log(en)$,
    so \Cref{lemma:delta} yields the bound
    $\delta(\K)\le 10 \log(en)$.
\end{proof}

\begin{proof}[Proof of \Cref{thm:delta-beta}]
    We proceed by induction.
    Let $\beta\ge 1$.
    Assume that
    \eqref{eq:bound-delta-beta}
    holds for this $\beta$.
    We now prove that
    \eqref{eq:bound-delta-beta}
    holds for $\beta+1$.
    Define the convex cones
    \begin{align}
      \coneleft_k = - \conebetak{k}
      \quad
      &\text{ if } k\ge \beta +1,
      \qquad
      \coneleft_{n-k} = \R^k \text{ otherwise,}
      \\
      \coneright_{n-k} = \conebetak{n-k},
      \quad
      &\text{ if } n-k\ge \beta +1,
      \qquad
      \coneright_k = \R^{n-k} \text{ otherwise,}
    \end{align}
    for all $k=1,...,n-1$,
    where $\conebetak{k}\subset\R^k$
    and $\conebetak{n-k}\subset\R^{n-k}$
    are defined in \Cref{example:beta}.
    Let $D_{n-\beta}, D_{n-\beta+1},...,D_n$ be the rectangular matrices defined in
    \eqref{eq:def-Dmatrix}.
    Let $\vtheta\in\conebetaat{\beta+1}$
    and
    define $\vv=D_{n-\beta+1}...D_n\vtheta$.
    As $\vtheta\in\conebetaat{\beta+1}$, $D_{n-\beta}\vv \ge0$.
    Thus $\vv$ is a nondecreasing sequence in $\R^{n-\beta}$.

Let $E=\{l=1,...,n-\beta: v_{l} \le 0\}$ where $v_l,l=1,...,n-\beta$ are the components of $\vv$.
    If $E$ is not empty, let $k=\beta-1+\max E$.
    Then $D_{k-\beta+1}...D_{k-1}D_k\vtheta_{S_k}\le 0$ so that $\vtheta_{S_k}\in - \conebetak{k}$,
    and $D_{n-k-\beta+1}...D_{n-k-1}D_{n-k}\vtheta_{T_k}\ge 0$ so that
    $\vtheta_{T_k}\in\coneright_{n-k}$. 
    If $E$ is empty, $\vv > \vzero$ and thus $\vtheta\in\conebeta$,
    so \eqref{eq:vtheta-Sk-Tk} holds for $k=1$.
    In summary, we have proved that for all $\vtheta\in\conebetaat{\beta+1}$,
    there exists $k=1,...,n-1$ such that \eqref{eq:vtheta-Sk-Tk} holds.

    Combining 
    \Cref{lemma:delta}
    with
    \eqref{eq:bound-delta-beta}
    yields
    \begin{align}
        \delta(\conebetaat{\beta+1})
        &\le 4 C(\beta) \log(en) + 6 \log(n-1), \\
        &\le (4 C(\beta) + 6)\log(en)
        =
        C(\beta+1)\log(en),
    \end{align}
    as by definition of $C(\beta+1)$ and $C(\beta)$,
    $4 C(\beta) + 6 = C(\beta+1)$.
\end{proof}

We now generalize \Cref{thm:increasings}
to the cones $\conebeta$ for $\beta\ge 1$.

\begin{thm}
    \label{thm:soi-for-all-beta}
    Let $\beta,n$ be integers such that $1 \le \beta < n$
    and let $\vmu\in\Rn$.
    Then
    \begin{equation}
        \Evmu
        \scalednorms{\hmu^{LS}(\conebeta) - \vmu}
        \le
        \min_{\vu\in\conebeta}
        \left(
            \scalednorms{\vu - \vmu}
            +
            \frac{C(\beta) \sigma^2 s_\beta(\vu)}{n}\log \frac{en}{s_\beta(\vu)}
        \right),
        \label{eq:soi-ls-Sbeta-E}
    \end{equation}
    where $C(\beta)$ depends only on $\beta$.
    Furthermore, for any $t>0$ we have
    \begin{equation}
        \scalednorms{\hmu^{LS}(\conebeta)- \vmu}
        \le
        \min_{\vu\in\conebeta}
        \left(
            \scalednorms{\vu - \vmu}
            +
            \frac{2C(\beta) \sigma^2 s_\beta(\vu)}{n}\log\frac{en}{s_\beta(\vu)}
        \right)
        + \frac{10 \sigma^2 t}{n}
        \label{eq:soi-ls-Sbeta-deviation}
    \end{equation}
    with probability greater than $1-\exp(-t)$.
\end{thm}
\begin{proof}[Proof of \Cref{thm:soi-for-all-beta}]
    Let $\vu\in\conebeta$ and let $k = s_\beta(\vu)$.
    Let $T_1,...,T_k$ be a partition of $\{1,...,n\}$ such that $\vu$ is a polynomial of degree $\beta-1$
    on all $T_j$, $j=1,...,k$.
    We apply \Cref{cor:soi-start} with 
    $\vomega= \vomega^{[\beta]}$,
    where $\vomega^{[\beta]}$ is defined in 
    \eqref{eq:vomega-beta}.
    Then
    inequality \eqref{eq:sharp-intermediate-simple} holds
    with $\K^\vomega_{|T_j|} = \conebetak{|T_j|}$.
    The rest of the proof is the same as the proof of \Cref{thm:increasings}
    with $a=1$,
    except that we use the bound \eqref{eq:bound-delta-beta}
    instead of \eqref{eq:upper-bounds-delta-increasing-log}.
\end{proof}

For $\beta=1$, the result above is exactly \Cref{thm:increasings} with $a=1$.
The following lower bound holds.

\begin{thm}
    \label{thm:lower}
    There exists an absolute constant $c>0$ such that the following holds.
    Let $\beta,s,n$ be positive integers such that $n\ge s$. 
    Then 
    \begin{equation}
    \label{eq:}
    \inf_{\hmu}
    \sup_{\vmu\in\conebeta:\; s_\beta(\vmu)\le s }
    \proba{
        \scalednorms{\hmu-\vmu}
        \ge \frac{c(\beta) s}{n}
    }
    \ge
    c,
    \end{equation}
    where the infimum is taken over all estimators
    and $c(\beta)>0$ is a constant that depends only on $\beta$.
\end{thm}
The proof of \Cref{thm:lower} is given in \Cref{s:lower}.
Under the assumption of \Cref{thm:lower}
for the integers $s,\beta$ and $n$,
Markov inequality yields
\begin{equation}
    \label{eq:lower-bound-minimax-risk}
    \inf_{\hmu}
    \sup_{\vmu\in\conebeta:\; s_\beta(\vmu)\le s }
    \Evmu
        \scalednorms{\hmu-\vmu}
        \ge \frac{c(\beta)c s}{n}.
\end{equation}
Consider the class
\begin{equation}
    \label{eq:def-sbeta-s}
    \conebeta(s)
    \coloneqq
    \{\vmu\in\conebeta: s_\beta(\vmu)\le s\}.
\end{equation}
The left hand side of
\eqref{eq:lower-bound-minimax-risk}
is the minimax
risk over this class.
We have proved that the minimax risk over this class
is of the order $\sigma^2 s/n$, up to a logarithmic factor.
To be more precise,
inequalities
\eqref{eq:soi-ls-Sbeta-E}
and
\eqref{eq:lower-bound-minimax-risk}
yield
\begin{equation}
    \frac{c(\beta)c \sigma^2 s}{n}
    \le
    \inf_{\hmu}
    \sup_{\vmu\in\conebeta(s)}
    \Evmu
    \scalednorms{\hmu-\vmu}
    \le
    \frac{C(\beta)\sigma^2 s \log(en/s)}{n}.
\end{equation}
Define the minimax regret as
\begin{equation}
    \inf_{\hmu}
    \sup_{\vmu\in\Rn}
    \left(
        \Evmu
        \scalednorms{\hmu-\vmu}
        - \min_{\vu\in\conebeta(s)}
        \scalednorms{\vu-\vmu}
    \right).
\end{equation}
Since the oracle inequality \eqref{eq:soi-ls-Sbeta-E} is sharp,
thus it implies the following bound on the maximal expected
regret of $\ls(\conebeta)$
\begin{equation}
    \label{eq:expected-regret}
    \sup_{\vmu\in\Rn}
    \left(
        \Evmu
        \scalednorms{\ls(\conebeta)-\vmu}
        - \min_{\vu\in\conebeta(s)}
        \scalednorms{\vu-\vmu}
    \right)
    \le
    \frac{C(\beta)\sigma^2 s \log(en/s)}{n}.
\end{equation}
Since the minimax risk is always smaller than the minimax regret,
the minimax regret also satisfies
\begin{equation}
    \label{eq:minimax-regret}
    \frac{c(\beta)c \sigma^2 s}{n}
    \le
    \inf_{\hmu}
    \sup_{\vmu\in\Rn}
    \left(
        \Evmu
        \scalednorms{\hmu-\vmu}
        - \min_{\vu\in\conebeta(s)}
        \scalednorms{\vu-\vmu}
    \right)
    \le
    \frac{C(\beta)\sigma^2 s \log(en/s)}{n}.
\end{equation}
For $\beta=1,2$,
this bracketing of the minimax regret was shown in \cite{bellec2015sharp}.
The estimator proposed in \cite{bellec2015sharp}
for which the upper bound of \eqref{eq:minimax-regret}
is attained
is an aggregate of an exponential number of estimators,
and cannot be computed in polynomial time.
In \eqref{eq:minimax-regret},
the upper bound on the minimax regret is attained at
the Least Squares estimator $\ls(\conebeta)$,
which can be computed efficiently by solving
a convex quadratic minimization program of size $n$.

\subsection{Cones of $m$-monotone sequences}

This section deals with
the cones defined in \Cref{example:monotone}.
Unlike the cones
$\increasings,\convexs$
and $\conebeta$ studied in the previous sections,
the lineality space of the cone $\mathcal M^m_n$
is $\{\vzero \}$ for all $n,m\ge1$.

Let $(T_1,...,T_k)$ be a partition of $\{1,...,n\}$.
To derive \Cref{thm:increasings},
we applied \Cref{cor:soi-start}
to the cones $\mathcal S^\uparrow_{|T_1|},...,\mathcal S^\uparrow_{|T_k|}$
and each of these cones has a lineality space of dimension 1.
Similarly, to derive
\Cref{thm:soi-for-all-beta}
we applied \Cref{cor:soi-start}
to the cones $\conebetak{|T_1|},...,\conebetak{|T_k|}$
and each of these cones has a lineality space of dimension $\beta$.
Although the lineality space of the cone $\mathcal M^m_n$ is $\{\vzero\}$
for all $m,n\ge 1$, \Cref{prop:soi-start} can be
used to derive the following oracle inequalities.

\begin{thm}
    \label{thm:m-monotone}
    Let $m,n$ be integers such that $1 \le m < n$
    and let $\vmu\in\Rn$.
    Then
    \begin{equation}
        \Evmu
        \scalednorms{\hmu^{LS}(\mathcal M^m_n) - \vmu}
        \le
        \min_{\vu\in\mathcal M^m_n}
        \left(
            \scalednorms{\vu - \vmu}
            +
            \frac{C(m) \sigma^2 s_m(\vu)}{n}\log \frac{en}{s_m(\vu)}
        \right),
    \end{equation}
    where $C(\cdot)$ is the constant from \Cref{thm:delta-beta}.
    Furthermore, for any $t>0$,
    \begin{equation}
        \scalednorms{\hmu^{LS}(\mathcal M^m_n)- \vmu}
        \le
        \min_{\vu\in\mathcal M^m_n}
        \left(
            \scalednorms{\vu - \vmu}
            +
            \frac{2C(m) \sigma^2 s_m(\vu)}{n}\log\frac{en}{s_m(\vu)}
        \right)
        + \frac{10 \sigma^2 t}{n}
    \end{equation}
    with probability greater than $1-\exp(-t)$.
\end{thm}
For $\vu\in\mathcal M^m_n$,
recall that 
$s_m(\vu)$ is the smallest number
$s$ such that $\vu$ is a piecewise polynomial function of degree at most $m-1$
with $s$ pieces (cf. \eqref{eq:polynomial}).

\begin{proof}[Proof of \Cref{thm:m-monotone}]
    Let $\vu\in\mathcal M^m_n$ and let $k = s_m(\vu)$.
    Let $T_1,...,T_k$ be a partition of $\{1,...,n\}$ such that $\vu$ is a polynomial of degree $m-1$
    on all $T_j$, $j=1,...,k$.
    Let $P_1,...,P_k$ be the coordinate projections
    \eqref{eq:def-coordinate-projections}.
    We apply \Cref{prop:soi-start} to the cones
    $\K_j \coloneqq \mathcal S^{[m]}_{|T_j|}$,
    $j=1,...,k$.
    For all $j=1,...,k$, $P_j\K\subset \K_j$ (cf. \eqref{eq:monotone-cap})
    and the restriction $\vu_{T_j}$
    is a polynomial of degree at most $m$.
    Thus,
    $\vu_{T_j} \in \lineality{\K_j}$.
    Then
    inequality \eqref{eq:sharp-intermediate} holds
    and can be rewritten
    in the form
    \begin{equation}
        \scalednorms{\hmu^{LS}(\mathcal M^m_n)- \vmu}
        \le
        \scalednorms{\vu - \vmu}
        +
        \sum_{j=1}^k
        \scalednorms{\Pi_{\mathcal S^{[m]}_{|T_j|}}(\vxi_{T_j})}.
    \end{equation}
    The rest of the proof is the same as the proof of \Cref{thm:increasings}
    with $a=1$,
    except that 
    instead of \eqref{eq:upper-bounds-delta-increasing-log}
    we use the bound \eqref{eq:bound-delta-beta}
    with $\beta$ replaced by $m$.
\end{proof}

\subsection{Non-Gaussian noise}

In this section, we do not assume that the noise vector $\vxi$
is normally distributed.
\Cref{prop:soi-start} does not depend on the distribution of the noise vector $\vxi$.
To illustrate this, we apply \Cref{cor:soi-start}
to the cone $\K=\increasings$.
For any $\vu\in\Rn$, let $(T_1,...,T_k)$ be a partition such that $\vu$ is constant
on all $T_j,j=1,...,k$.
Taking expectations of both sides of
\eqref{eq:sharp-intermediate-simple} yields
\begin{equation}
    \E \scalednorms{\ls(\increasings) -\vmu}
    \le
    \scalednorms{\vu -\vmu}
    + \sum_{j=1}^k
    \E \scalednorms{\Pi_{\mathcal S^\uparrow_{|T_j|}}(\vxi_{T_j})}.
    \label{eq:start-non-gaussian}
\end{equation}
Let  $\vxi=(\epsilon_1,...,\epsilon_N)^T$
where $\epsilon_1,...,\epsilon_N$ are i.i.d.
random variables
with $\E\epsilon_1 = 0$
and
$\E[\epsilon_1^2] \le \sigma^2$.
It was shown in \cite[Theorem 3.1]{chatterjee2013risk}
that for all $N\ge 1$,
\begin{equation}
    \E \scalednorms{\Pi_{\mathcal S^\uparrow_N}(\vxi)} \le 4\sigma^2\log(eN).
    \label{eq:bound-non-gaussian}
\end{equation}
Combining this bound with \eqref{eq:start-non-gaussian}
and Jensen's inequality yields the following result.
\begin{cor}
    Let $\vmu\in\Rn$.
    Let $\vxi=(\epsilon_1,...,\epsilon_n)^T$
    where $\epsilon_1,...,\epsilon_N$ are i.i.d. with $\E\epsilon_1 = 0$
    and
    $\E[\epsilon_1^2] = \sigma^2$. Then for all $\vu\in\increasings$,
    \begin{equation}
        \E \scalednorms{\ls(\increasings) -\vmu}
        \le
        \scalednorms{\vu -\vmu}
        + 
        \frac{4\sigma^2 k(\vu) \log(en/k(\vu))}{n}.
    \end{equation}
\end{cor}

\subsection{Multivariate isotonic regression}

\Cref{prop:soi-start} is not limited to univariate regression.
Let $d>1$, let $n_1,...,n_d > 1$ be integers and let $n=n_1n_2...n_d$.
Consider the discrete hyperrectangle
\begin{equation}
    \label{eq:def-I}
    I = \{ (i_1,...,i_d)\subset \mathbf{N}^d, \quad 1\le i_l\le n_l, \text{ for all } l=1,...,d\},
\end{equation}
and the cone $\K^{d \uparrow}\subset \mathbf{R}^I$ defined by
\begin{multline}
    \K^{d \uparrow}
    = \Big\{
            \vu=(
            u_{i_1i_2...i_d}
        )_{(i_1,i_2,...,i_d)\in I}
                \in \mathbf{R}^I
            ,
        \\
        \text{such that }
            u_{i_1i_2...i_d}
            \le
            u_{j_1j_2...j_d}
            \text{ if }
            (i_l\le j_l \text{ for all }l=1,...,d)
    \Big\}.
    \label{eq:def-tensor}
\end{multline}
The set
$\K^{d \uparrow}$
is the cone of vectors indexed by $I$
that are nondecreasing in all directions $l=1,...,d$.
For $d=2$, the performance of the Least Squares estimator over
$\K^{2 \uparrow}$ has been recently studied in \cite{chatterjee2015matrix}.
Let $k$ be a positive integer.
Consider a partition $(T_1,...,T_k)$ of $I$ 
and $\vu\in \K^{d\uparrow}$ such that $\vu$ is constant on $T_j$ for all $j=1,...,k$.
Then, for any unknown $\vmu\in\mathbf{R}^I$,
\Cref{prop:soi-start} yields
\begin{equation}
    \Evmu \scalednorms{\ls(\K^{d\uparrow}) - \vmu}
    \le
    \scalednorms{\vu - \vmu}
    + \frac{\sigma^2}{n} \sum_{j=1}^k \delta_j,
    \label{eq:higher-d}
\end{equation}
where for all $j=1,...,r$,
$\delta_j$ is the statistical dimension of the cone
\begin{multline}
    \Big\{
            \vu=(
            u_{i_1i_2...i_d}
        )_{(i_1,i_2,...,i_d)\in T_j}
        \in \mathbf{R}^{T_j}
        \\
        \text{such that }
            u_{i_1i_2...i_d}
            \le
            u_{j_1j_2...j_d}
            \text{ if }
            (i_l\le j_l \text{ for all }l=1,...,d)
    \Big\}.
\end{multline}
If $d=2$ and $T_1,...,T_k$ are rectangles,
\citet{chatterjee2015matrix} proved that
$\delta_j \le C \log(en)^8$ for some absolute constant $C$.
In that case, a direct consequence of \Cref{prop:soi-start} is
\begin{equation}
    \Evmu \scalednorms{\ls(\K^{2\uparrow}) - \vmu}
    \le
    \scalednorms{\vu - \vmu}
    + \frac{C \sigma^2 k \log(en)^8 }{n}.
    \label{eq:sharp-tensor-2}
\end{equation}
This improves upon
the oracle inequality \cite[Theorem 4.1]{chatterjee2015matrix}
that has a leading constant strictly greater than 1.
More importantly, \eqref{eq:higher-d} above
shows that our method does not depend on the underlying dimension
since \eqref{eq:higher-d} holds for any $d\ge2$.

\section{OLD}

\subsection{Orthogonal decomposition and lineality spaces}
\Cref{prop:soi-start} below 
is our main tool to derive sharp oracle inequalities 
for the estimator $\ls(\K)$ for any closed convex cone $\K$.
Given a matrix $P$,
denote by $\Ima P$ the linear span of the columns of $P$.

\begin{prop}
    \label{prop:soi-start}
    Let $n\ge 2$, $\vmu\in\Rn$,
    let $\K$ be a closed convex set 
    and
    let $\vu\in\K$. Furthermore, assume \ref{list:i} and \ref{list:ii} below.
    \begin{enumerate}[label=(\roman*)]
        \item \label{list:i}
            There exist orthogonal projectors $P_1,...,P_k$ 
            such that
            \begin{align}
                \sum_{j=1}^k P_j = I_{n\times n}
                \qquad
                \text{ and }
                \qquad
                P_j P_l = 0 \quad\text{for all }j,l=1,...,k.
            \end{align}
        \item \label{list:ii}
        There exist closed convex cones $\K_1,...,\K_k$ such that
        \begin{equation}
            \{P_j\vv,\vv\in\K\} \subseteq \K_j \subseteq \Ima P_j
            \quad
            \text{and}
            \quad
            P_j \vu \in \lineality{\K_j},
            \qquad j=1,...,k.
        \end{equation}
    \end{enumerate}
    Then almost surely
    \begin{equation}
        \scalednorms{\hmu^{LS}(\K) - \vmu}
        \le
        \scalednorms{\vu - \vmu}
        +
        \sum_{j=1}^k
        \scalednorms{\Pi_j(P_j\vxi)}
        ,
        \label{eq:sharp-intermediate}
    \end{equation}
    where
    $\Pi_j: \Ima P_j\rightarrow \K_j$
    is
    the projection onto 
    $\K_j$, $j=1,...,k$.
\end{prop}

The assumptions of \Cref{prop:soi-start}
on the projectors $P_1,...,P_k$ imply
\begin{equation}
    \Ima P_1
    \oplus
    \dots
    \oplus
    \Ima P_k
    =
    \Rn,
\end{equation}
where 
$\oplus$ denotes an orthogonal direct sum.
The random variables $P_1\vxi,...,P_k\vxi$
are thus independent normal random variables.

\Cref{prop:soi-start} allows us
to bound from above the loss of $\ls(\K)$
by bounding from above the sum
of the 
$k$ independent
random variables $\scalednorms{\Pi_1(P_1\vxi)},...,\scalednorms{\Pi_k(P_k\vxi)}$.
By the definition
 of the statistical dimension of a cone given
in \eqref{eq:def-delta},
\Cref{prop:soi-start} shows that 
upper bounds on the statistical dimensions of the cones
$\K_{1},...,\K_{k}$
imply a sharp oracle inequality for the estimator $\ls(\K)$.

If $\K$ is a closed convex cone,
a natural choice for $\K_1,...,\K_k$
is $\K_j=\{P_j\vv,\vv\in\K\}$,
$j=1,...,k$.
However, we will see
in \Cref{thm:m-monotone}
an application of \Cref{prop:soi-start}
with a different choice for the cones $\K_1,...,\K_k$,
and in 
\eqref{eq:oi-subset}
an application of \Cref{prop:soi-start}
where $\K$ is not a cone.

To prove \Cref{prop:soi-start},
we use the strong convexity of the Least Squares minimization problem,
as follows.

\begin{proof}[Proof of \Cref{prop:soi-start}]
    Let $\hmu = \hmu^{LS}(\K)$ for notational simplicity.
    Inequality \eqref{eq:kkt-noise}
    can be rewritten as
    \begin{equation}
        \euclidnorms{\hmu - \vmu}
        - \euclidnorms{\vu - \vmu}
        \le 2 \vxi^T(\hmu - \vu) - \euclidnorms{\vu - \hmu}
        = \euclidnorms{\vxi} - \euclidnorms{\vxi - \hmu + \vu}.
        \label{eq:kkt-ls}
    \end{equation}
    For all $j=1,...,k$, we have $P_j\hmu\in\K_j$
    and $P_j\vu\in\lineality{\K_j}$,
    so $P_j(\hmu-\vu)\in\K_j$.
    Thus, by definition of $\Pi_1,...,\Pi_k$,
    \begin{align}
        \euclidnorms{\vxi} - \euclidnorms{\vxi - \hmu + \vu}
        & =
        \sum_{j=1}^k \euclidnorms{P_j\vxi} - \euclidnorms{P_j(\vxi - (\hmu-\vu))}, \\
        & \le
        \sum_{j=1}^k \euclidnorms{P_j\vxi} - \euclidnorms{P_j\vxi - \Pi_j(P_j\vxi)}, \\
        & =
    \sum_{j=1}^k 2 (P_j\vxi)^T \Pi_j(P_j\vxi) - \euclidnorms{\Pi_j(P_j\vxi)}
        =
        \sum_{j=1}^k \euclidnorms{\Pi_j(P_j\vxi)},
    \end{align}
    where for the last equality we used that if $\Pi$ is a projection onto a closed convex cone,
    $\Pi(\vtheta)^T(\vtheta - \Pi(\vtheta)) = 0$ for all vectors $\vtheta$ (cf. \eqref{eq:Pi_characterisation}).
    By plugging the previous display back into \eqref{eq:kkt-ls}
    and dividing by $n$,
    we obtain \eqref{eq:sharp-intermediate}.
\end{proof}

For any $T\subset\{1,...,n\}$ and $\vv\in\Rn$, denote by
$\vv_T\in\R^{|T|}$ the restriction of $\vv$ to the set $T$
and by $|T|$ the cardinality of $T$.
Let $(T_1,...,T_k)$ be a partition of $\{1,...,n\}$
and let $P_1,...,P_k$
be the coordinate projections
\begin{equation}
P_j = \sum_{l\in T_j} \ve_l \ve_l^T,
\qquad
\text{for all }
\quad
j=1,...,k.
\label{eq:def-coordinate-projections}
\end{equation}
If $\K=\K^\vomega_n$ for some vector $\vomega\in\Rm$ (cf. \Cref{example:weights}),
and the cones $\K_1,...,\K_k$ are given by
$\K_j = P_j \K$ for all $j=1,...,k$, then
$\K_j = \K^\vomega_{|T_j|}$ and \Cref{prop:soi-start}
takes the following form.

\begin{cor}
    \label{cor:soi-start}
    Let $\vomega\in\Rm$ for some $m\le n$.
    Let $(T_1,...,T_k)$ be a partition of $\{1,...,n\}$
    such that
    for all  $j=1,...,k$,
    $T_j$ has the form $T_j = \{t_j+1,...,t_j+|T_j|\}$ for some integer $t_j\ge0$.
    Let $\vu=(u_1,...,u_n)^T\in\K^\vomega_n$ be such that
    \begin{equation}
        \vu_{T_j} \in \lineality{\K^\vomega_{|T_j|}},
        \qquad
        j=1,...,k.
    \end{equation}
    Then, almost surely
    \begin{equation}
        \scalednorms{\ls(\K^\vomega_n) -\vmu}
        \le
        \scalednorms{\vu -\vmu}
        + \sum_{j=1}^k
        \scalednorms{\Pi_{\K^\vomega_{|T_j|}}(\vxi_{T_j})}.
        \label{eq:sharp-intermediate-simple}
    \end{equation}
\end{cor}

To illustrate \Cref{prop:soi-start} and \Cref{cor:soi-start},
we now prove \Cref{thm:increasings}.
\begin{proof}[Proof of \Cref{thm:increasings}]
    Let $\hmu = \hmu^{LS}(\increasings)$ for notational simplicity.
    Let $\vu\in\increasings$ and let $k = k(\vu)$.
    Let $T_1,...,T_k$ be a partition of $\{1,...,n\}$ such that $\vu$ is constant on all $T_j$, $j=1,...,k$.
    Thanks to \Cref{cor:soi-start} with $\vomega=(-1,1)^T$,
    inequality \eqref{eq:sharp-intermediate-simple} holds
    where $\K^\vomega_{|T_j|} = \mathcal S^\uparrow_{|T_j|}$.
    We will first prove \eqref{eq:soi-ls}.
    It was shown in  \cite[Appendix D.4]{amelunxen2014living}
    that if $\vg\sim\mathcal N (\vzero,I_{n\times n})$, 
    then
    \begin{equation}
        \E \left[ \euclidnorms{\Pi_{\mathcal S^\uparrow_{|T|}}(\vg_{T})} \right] = \sum_{l=1}^{|T|} \frac{1}{l} \le \log(e |T|),
        \label{eq:upper-bounds-delta-increasing-log}
    \end{equation}
    where $T=\{1,...,n\}$.
    To complete the proof of \eqref{eq:soi-ls},
    we take the expectations in \eqref{eq:sharp-intermediate-simple}
    and apply Jensen's inequality to get
    \begin{equation}
        \sum_{j=1}^k
        \log(e|T_j|)
        =
        k
        \sum_{j=1}^k
        \frac{1}{k}
        \log(e|T_j|)
        \le 
        k
        \log\left(
            \frac{e}{k} \sum_{j=1}^k |T_j|
        \right)
        =
        k \log\frac{en}{k}.
        \label{eq:concavity-increasings}
    \end{equation}
    To prove \eqref{eq:soi-ls-deviation},
    we use \eqref{eq:sharp-intermediate-simple} where 
    $\vu\in\increasings$ is a minimizer of the right hand side of
    \eqref{eq:soi-ls-deviation},
    and then apply \Cref{lemma:concentration-proj} to the stochastic term.
\end{proof}

As a consequence of \Cref{prop:soi-start},
we can derive sharp oracle inequalities for the Least Squares estimator
if we can bound from above the statistical dimension
of certain cones.

We now illustrate how \Cref{prop:soi-start}
can be used in situations where $\K$ is not a cone.
Let $\K$ be any closed convex subset of $\increasings$.
Let $\vu\in\K$ and let $(T_1,...,T_k)$ be a partition of $\{1,...,n\}$
such that $\vu$ is constant on all $T_j,j=1,...,k$.
Let $P_1,...,P_k$ be the coordinate projections \eqref{eq:def-coordinate-projections} and let $\K_j = \mathcal S^\uparrow_{|T_j|}$.
Applying \eqref{eq:sharp-intermediate}
and
following the same arguments as in the proof of \Cref{thm:increasings},
we obtain that for any closed convex subset $\K$ of $\increasings$,
\begin{equation}
    \Evmu
    \scalednorms{\hmu^{LS}(\K)- \vmu}
    \le
    \min_{\vu\in\K}
    \left(
        \scalednorms{\vu - \vmu}
        +
        \frac{\sigma^2 k(\vu)}{n}\log\frac{en}{k(\vu)}
    \right).
    \label{eq:oi-subset}
\end{equation}
For instance, \eqref{eq:oi-subset} holds
for
$\K =\{ \vu\in\increasings: a^- \le u_1, u_n \le a^+\}$
where $a^-<a^+$ are fixed real numbers.

\section{Concluding remarks}
We have presented two general methods to derive
sharp oracle inequalities
for the Least Squares estimator
over a closed convex subset of $\Rn$.
First, \Cref{prop:soi-start} shows that the Least Squares estimator
over a closed convex set
satisfies a sharp oracle inequality in deviation and expectation,
where the remainder term is proportional to the sum of the
statistical dimensions of some cones (cf. \eqref{eq:sharp-intermediate}).
The second method is based
on localized Gaussian widths and is given in
\Cref{thm:isomorphic}.
If $\C$ is a closed convex subset of $\Rn$,
\Cref{thm:isomorphic}
shows
that the Least Squares estimator
$\ls(\C)$ satisfies a sharp oracle inequality in deviation and expectation
if the 
localized Gaussian width of $\C$ satisfies condition
\eqref{eq:fixed-point-C}
for some constant $t_*>0$.
To summarize, our methods lead to the following improvements.
\begin{enumerate}[label=(\roman*)]
    \item 
        Our oracle inequalities hold 
        not only for the expected squared risk,
        but
        also in deviation with exponential probability tails.
        By integration, a sharp oracle inequality in deviation with exponential probability tails
        always implies a sharp oracle inequality in expectation.
        The reverse is not true, as there exist estimators
        that satisfy sharp oracle inequalities in expectation but not in deviation \cite{audibert2007no,dai2012deviation}.
    \item
        Another improvement of the present paper over \cite{zhang2002risk,guntuboyina2013global,chatterjee2013risk,chatterjee2015matrix}
        is that our oracle inequalities are sharp, i.e., with leading constant 1.
        Thus, our bounds account for model misspecification.
        This advantage can be interpreted at least in the following two ways.
        \begin{enumerate}
            \item 
                Let $\C$ be a closed convex set such that $\hmu\notin\C$ 
                and $\hmu$  an estimator valued in $\C$.
                The quantity 
                $\scalednorms{\hmu - \Pi_\C(\vmu)}$
                is a natural measure of the performance of $\hmu$.
                As seen in \eqref{eq:soi-imply},
                sharp oracle inequalities grant upper bounds
                on the quantity $\scalednorms{\hmu - \Pi_\C(\vmu)}$,
                whereas oracle inequalities with leading constant 
                strictly greater than 1 do not.
            \item 
                A second advantage of sharp oracle inequalities
                is that they allow to bound from above the minimax regret.
                To see this,
                let $E,\bar E$ be two subsets of $\Rn$ with $E \subset \bar E$.
                If  $\vmu\in E$,
                we say that the model is well-specified, if $\vmu\in\bar E\setminus E$ the model is misspecified.
                The minimax risk over $E$ is
                $\inf_\tildemu
                \sup_{\vmu\in E}
                \Evmu \scalednorms{\tildemu - \vmu}$
                and the minimax regret with respect to $(E,\bar E)$ is
                \begin{equation}
                    \inf_\tildemu
                    \sup_{\vmu\in \bar E}
                    \left(
                        \Evmu \scalednorms{\tildemu - \vmu}
                        - \inf_{\vu\in E}\scalednorms{\vu- \vmu}
                    \right),
                \end{equation}
                where the infima are taken over all estimators.
                The minimax risk is a measure of the statistical complexity of $E$
                if the model is well-specified.
                The minimax regret is a natural measure of the statistical complexity of $E$
                that accounts for misspecification with respect to the set $\bar E$.
                There
                are situations where the minimax regret is substantially greater
                than the minimax risk
                \cite{rakhlin2013empirical}.
                Thus, it is important to study both the minimax risk and the minimax regret.
                As follows from 
                \eqref{eq:expected-regret}
                and
                \eqref{eq:expected-regret2},
                the
                sharp oracle inequalities \eqref{eq:soi-ls-Sbeta-E}
                and \eqref{eq:soi-increasing-uniform}
                yield upper bounds on the minimax regret for $(E,\bar E)=(\conebeta(s),\Rn)$ and $(E,\bar E)=(\increasings,\{\vmu\in\Rn:\inftynorm \vmu \le D\})$.
                On the other hand,
                risk bounds or oracle inequalities with leading constant strictly greater than 1
                do not imply bounds on the minimax regret.
        \end{enumerate}

\end{enumerate}

\section{Application of isomorphic - WIP}

\begin{cor}
    \label{cor:rate12}
    There exist absolute constants $C>0$ such that the following holds.
    Let $d=2$ and $n = n_1n_2$ for two positive integers $n_1,n_2$.
    Let $\vmu\in\mathbf{R}^I$ where $I$ is defined in \eqref{eq:def-I},
    and let $\vmu^*$ be the projection of $\vmu$ onto
    the cone $\K^{2\uparrow}$ defined in \eqref{eq:def-tensor}.
    Then for all $x>0$,
    with probability greater than $1-\exp(-x)$,
    \begin{equation}
        \scalednorms{
            \ls(\K^{2\uparrow})
            -
            \vmu
        }
        \le
        \min_{\vu\in\K^{2\uparrow}}
        \scalednorms{
            \vu
            -
            \vmu
        }
    +   
            \frac{C\sigma^2 \log(en)^8}{n}
            +
            \frac{C\sqrt{\sigma^2 V(\vmu^*)} }{n^{1/2}}
        + \frac{16\sigma^2 x}{n},
    \end{equation}
    where $V(\vmu^*) = (1/n) 
    \sum_{i_1=1}^{n_1}
\sum_{i_2=1}^{n_2}
(\mu^*_{i_1i_2} - \bar\mu^*)^2$
and
$
\bar\mu^* =
(1/n)
    \sum_{i_1=1}^{n_1}
\sum_{i_2=1}^{n_2}
\mu^*_{i_1i_2}
$.
\end{cor}
\begin{proof}[Proof of \Cref{cor:rate12}]
    It was proved in \cite[Section 2]{chatterjee2015matrix}
    that there exists an absolute constant $c>0$ 
    such that $t_*\coloneqq c \sqrt{\sigma^2\log(en)^8/n + \sqrt{\sigma^2 V(\vmu^*)/n}}$ satisfies
    \eqref{eq:fixed-point-C}.
    Applying \Cref{thm:isomorphic} completes the proof.
\end{proof}

For any $\vmu=(\mu_1,...,\mu_n)^T\in\Rn$, let  $\inftynorm{\vmu} = \max_{i=1,...,n}|\mu_i|$.
It is easy to see that $\inftynorm{\Pi_\increasings(\vmu)}\le \inftynorm{\vmu}$.
By integration, 
\eqref{eq:soi-increasing-uniform} implies the following bound on the
maximal expected regret of $\ls(\increasings)$
over the class $\{\vmu\in\Rn:\inftynorm{\vmu}\le D\}$
\begin{equation}
    \label{eq:expected-regret2}
    \sup_{\vmu\in\Rn:\inftynorm{\vmu}\le D}
    \left(
        \Evmu
        \scalednorms{
            \ls(\increasings)
            -
            \vmu
        }
        -
        \min_{\vu\in\increasings}
        \scalednorms{
            \vu
            -
            \vmu
        } 
    \right)
    \le
    C' \left(\frac{ D \sigma^2}{n}\right)^{2/3},
\end{equation}
where $D\ge \sigma$ is a fixed parameter
and $C'>0$ is an absolute constant.
Similar regret bounds hold for the estimator
$\ls(\convexs)$ with the rate $n^{-4/5}$,
and 
for $\ls(\K^{2\uparrow})$ with the rate $n^{-1/2}$.

\section{Proofs: Lower bound}
\label{s:lower}

Define the support of $\vv=(v_1,...,v_n)^T\in\Rn$
by $\supp(\vv) = \{i=1,...,n: v_i\ne 0\}$.

\begin{proof}[Proof of \Cref{thm:lower}]
    First, assume that 
    $s \ge 9 \beta + 1 $.
    Let $S\ge 8$ be the largest integer
    such that $(S+1)\beta + 1\le s$.
    For all
    $\vomega = (\omega_1,...,\omega_S)^T\in \{0,1\}^S$
    and $\vomega' = (\omega_1',...,\omega_S')^T\in \{0,1\}^S$,
    define the Hamming distance between $\vomega$ and $\vomega'$ by 
    $d_H(\vomega,\vomega') = \sum_{k=1}^S |\omega_k - \omega_k'|$.
    By the Varshamov-Gilbert bound \cite[Lemma 2.9]{tsybakov2009introduction},
    there exists $\Omega\subseteq \{0,1\}^S$ such that
    \begin{equation}
        \vzero = (0,...,0)^T\in\Omega,
        \quad
        \log (|\Omega| - 1)\ge S/8,
        \quad
        \text{and}
        \quad
        d_H(\vomega,\vomega') > S/8
    \end{equation}
    for all $\vomega,\vomega'\in\Omega$ such that $\vomega\ne \vomega'$.
    Let $m$ be the largest integer such that $mS \le n$.
    For each $\vomega\in\Omega$, define
    $\vu^\vomega=(u^\vomega_1,...,u^\vomega_n)$
    by
    \begin{equation}
        \qquad
        u_i^\vomega = \omega_j
        \qquad
        \text{ if } jm \le i -1 < jm+1,
        \qquad
        \text{ for } i=1,...,Sm\text{ and } j=1,...,S,
    \end{equation}
    and $u_i^\vomega=0$ if $i>Sm$.
    Let $\Delta=\design^{-1}$ where $\design$ is the matrix \eqref{eq:design},
    i.e., $\Delta\vu = (u_1, u_2-u_1, ..., u_n-u_{n-1})^T$ for all $\vu=(u_1,...,u_n)^T\in\Rn$.
    For all $\vomega\in\Omega$,
    $\vu^\vomega$ is piecewise constant with at most $S+1$ pieces.
    It is easy to see that
    $\Delta\vu^\vomega$ has at most $S+1$ nonzero components
    and 
    \begin{equation}
        \supp(\Delta\vu^\vomega)\subset\cup_{k=0}^S\{km + 1\}.
    \end{equation}
    An immediate recurrence yields that for all $\vomega\in\Omega$,
    \begin{equation}
        \supp(\Delta^\beta\vu^\vomega)\subset\cup_{k=0}^S\{km + 1, km+2,...,km+\beta\}.
    \end{equation}
    Let $T\coloneqq (\cup_{k=0}^S \{km + 1, km+2,...,km+\beta\}) \cap\{1,...,n\}$.
    We have $|T|\le (S+1)\beta$ and $\supp(\Delta^\beta\vu^\vomega)\subset T$ for all $\vomega\in\Omega$.
    Define $\vtheta=(\theta_1,...,\theta_n)^T\in\Rn$
    by
    \begin{equation}
        \theta_i = \max\left(0, \max_{\vomega\in\Omega} \left( -(\Delta^\beta \vu^\vomega)_i \right) \right)
        \quad
        \text{ if }
        i\in T,
        \qquad
        \theta_i = 0
        \quad
        \text{ if } i\notin T.
    \end{equation}
    By construction, for all $\vomega\in\Omega$,
    $\supp(\vtheta + \Delta^\beta \vu^\vomega)\subset T$ 
    and $\vtheta + \Delta^\beta \vu^\vomega$ has nonnegative entries.
    Let $\gamma = (1/8)\sqrt{\sigma^2/m}$.
    Using \Cref{lemma:sbeta} below,
    for all $\vomega\in\Omega$,
    \begin{equation}
        \vx^\vomega \coloneqq \gamma \design^\beta\left(
            \vtheta + \Delta^\beta \vu^\vomega
        \right)
        = \gamma \design^\beta \vtheta + \gamma \vu^\vomega
    \end{equation}
    belongs to 
    $
        \conebeta$,
    and $
    s_\beta(\vx^\vomega) \le |T| + 1 \le \beta(S+1) +1 \le s$.
    For two distinct $\vomega,\vomega'\in\Omega$, we have
    \begin{equation}
    \label{eq:}
    \scalednorms{\vx^\vomega - \vx^{\omega'}}
    =
    \frac{\gamma^2 d_H(\vomega,\vomega')m}{n}
    \ge
    \frac{\gamma^2 S m}{8n}
    \ge
    \frac{\gamma^2}{16},
    \end{equation}
    as by definition of $m$,  $n/(2S)< m \le n/S$.
    Denote by $P_\vomega$ the distribution of $\vx^\vomega + \vxi$.
    For any $\vomega\in\Omega$, 
    the Kullback-Leibler divergence between the measures $P_\vomega$ and $P_\vzero$
    satisfies
    \begin{equation}
        K(P_\vomega,P_\vzero)
        =
        \frac{n}{2\sigma^2}\scalednorms{\vx^\vomega - \vx^\vzero}
        =
        \frac{n\gamma^2}{2\sigma^2}\scalednorms{\vu^\vomega - \vu^\vzero}
    \le \frac{\gamma^2 Sm}{2\sigma^2}
    \le \frac{S}{128}
    \le \frac{\log|\Omega| -1}{16}.
    \end{equation}
    By \cite[Theorem 2.7]{tsybakov2009introduction} with $\alpha = 1/16$,
    there exists an absolute constants $c>0$ such that
    \begin{equation}
    \label{eq:}
    \inf_{\hmu}
    \sup_{\vmu\in\conebeta:\; s_\beta(\vmu)\le s }
    \proba{
        \scalednorms{\hmu-\vmu}
        \ge \frac{\gamma^2}{64}
    }
    \ge
    c.
    \end{equation}
    By definition of $S$ we have $s\le 2 S \beta$,
    and by definition of $m$ we have $Sm\le n$.
    This implies that
    \begin{equation}
        64\gamma^2 =
        \frac{ \sigma^2}{m} 
        \ge
        \frac{\sigma^2 S}{n}
        \ge
        \frac{\sigma^2 s}{2\beta n}.
    \end{equation} 

    It remains to consider the case $s < 9\beta +1$.
    In this case, the rate is of order $1/n$ and the lower bound follows from standard arguments
    by a reduction to testing between two simple hypotheses.
\end{proof}

\begin{lemma}
    \label{lemma:sbeta}
    Let $n<\beta$ be positive integers
    and let $\design$ be the matrix \eqref{eq:design}.
    Assume that $\vtheta=(\theta_1,...,\theta_n)^T\in\Rn$ has nonnegative entries,
    i.e., $\theta_i\ge0,i=1,...,n$.
    Then
    $\design^\beta\vtheta\in\conebeta$ and
    $s_\beta(\design^\beta\vtheta)
    \le  |\supp(\vtheta)|+1$.
\end{lemma}
\begin{proof}
    The claim $\design^\beta\vtheta\in\conebeta$
    follows from
    \begin{equation}
        D_{n-\beta+1}D_{n-\beta + 2}...D_n\design^\beta \vtheta = (\theta_{\beta+1},...,\theta_n)^T.
    \end{equation}
    Let $t_1=1$.
    There exists $k\ge 1$ such that
    $\supp(\vtheta)\cup\{t_1\}=\{t_1,...,t_{k}\}$ with $t_1<...<t_{k}$
    and $k\le |\supp(\vtheta)| + 1$.
    Let $t_{k+1}=n+1$
    and define a partition $(T_1,...,T_k)$ of $\{1,...,n\}$
    by $T_j=\{t_j ,...,t_{j+1}-1\},j=1,...,k$.
    Let $\vu_j = (\design^\beta\vtheta)_{T_j} \in \R^{|T_j|}$.
    Let $j=1,...,k$.
    If $|T_j|\le \beta$ then using interpolation polynomials,
    the vector
    $\vu_j$ satisfies $(\vu_j)_i = Q_j(i),i=1,...,|T_j|$
    for some polynomial $Q_j$ of degree at most $\beta - 1$.
    If $|T_j|>\beta$ then
    \begin{equation}
        D_{|T_j|-\beta+1}...D_{|T_j|} \vu_j
        = (\theta_{t_{j}+\beta},...,\theta_{t_{j+1} - 1})^T.
    \end{equation}
    By definition of $t_1,...,t_k$ we have $(\theta_{t_{j}+\beta},...,\theta_{t_{j+1} - 1})^T = \vzero$.
    Thus $\vu_j\in\lineality{\conebetak{|T_j|}}$ and there exists a polynomial $Q_j$ of degree at most $\beta-1$
    such that $(\vu_j)_i = Q_j(i), i=1,...,|T_j|$. 
    We have established that 
    $s_\beta(\design^\beta\vtheta)
    \le k \le  |\supp(\vtheta)|+1$.
\end{proof}

\section{OLD}

If $\K$ is a closed convex cone,
we will show that the estimator $\hmu=\ls(\K)$ satisfies
oracle inequalities of the form
\begin{equation}
    \label{eq:example-oi-adapt}
    \Evmu
    \scalednorms{\hmu-\vmu}
    \le
    C
    \min_{\vu\in \K}
    \left(
        \scalednorms{\vu-\vmu}
        +
        r_{\sigma,n}(\vu)
    \right),
\end{equation}
where $C\ge1$ and
the remainder term
$
r_{\sigma,n}(\vu)
$
depends on $\vu$, unlike
\eqref{eq:example-oi} where the remainder term depends on $\K$ but not on $\vu$.
The right hand side of 
\eqref{eq:example-oi-adapt} is minimized
by a vector $\vu\in \K$ that makes a trade-off
between the approximation error
$\scalednorms{\vu-\vmu}$ and the quantity
$r_{\sigma,n}(\vu)$.
Bounds 
such as \eqref{eq:example-oi-adapt} will be called adaptive oracle inequalities.
Such trade-off is common in the context of sparse linear regression,
where the right hand side of sparsity oracle inequalities
balances approximation error and sparsity 
\cite{bickel2009simultaneous,koltchinskii2011nuclear,rigollet2011exponential}.
This is opposed to the oracle inequality
\eqref{eq:example-oi} where there is no trade-off,
the right hand side of \eqref{eq:example-oi}
is always minimized for $\vu=\Pi_\K(\vmu)$
which achieves the smallest approximation error.

\section{Other representations}
\subsection{increasing}
For all $q\ge 2$, denote by $D_q$ the following matrix with $q-1$ rows and $q$ columns
\begin{equation}
    D_q \coloneqq
\begin{bmatrix}
        -1 & 1 & 0 & \dots & \dots  & 0  \\
        0 & -1 & 1 & \dots & \dots  & 0 \\
        \vdots & \vdots & \ddots & \ddots & \ddots & \vdots \\
    0 & \dots & \dots & 0 & -1  & 1
\end{bmatrix}
.
    \label{eq:def-Dmatrix}
\end{equation}
For all $q\ge 2$,
denote by $\le$ and $\ge$ the component-wise comparison operators in $\R^q$, i.e.,
\begin{equation}
    (u_1,...,u_q)^T
    \le
    (v_1,...,v_q)^T
    \quad
    \text{ if and only if }
    \quad
    \forall i=1,...,q,
    \quad
    u_i \le v_i.
\end{equation}
\begin{equation}
    \increasings
    =
    \{\vu=(u_1,...,u_n)^T\in\Rn:  D_n \vu \ge \vzero = (0,...,0)^T \},
\end{equation}
where $D_n$ is the matrix \eqref{eq:def-Dmatrix}.
Define the matrix 
$\design=(\design_{ij})_{i=1,...,n,\;j=1,...,n}$
by
\begin{equation}
    \label{eq:design}
    \design_{ij}= 1 \quad\text{ if } j\le i \qquad \text{and} \qquad \design_{ij}= 0 \quad \text{ otherwise}.
\end{equation}
Then we can represent the set $\increasings$
as a linear model with design matrix $\design$ 
and nonnegative coefficients:
\begin{equation}
\label{eq:def-increasings-design}
    \increasings =
    \left\{
        \design \vtheta,
        \qquad
        \vtheta=(\theta_1,...,\theta_n)^T\in\Rn:
        \quad
        \theta_k \ge 0 \text{ for all } k\ge 2
    \right\}.
\end{equation}
If $\vu = \design \vtheta$ with $\vtheta$ as in \eqref{eq:def-increasings-design},
then $k(\vu)$ is also the number of strictly positive entries among $\theta_2,...,\theta_n$.

\subsection{Convex}
If $D_{n-1}$ and $D_n$ are the rectangular matrices defined in
\eqref{eq:def-Dmatrix} then
\begin{align}
    \convexs 
    = \{ \vu=(u_1,\dots,u_n)^T\in\Rn: \; D_{n-1}D_n\vu \ge \vzero = (0,...,0)^T \},
\end{align}
If $\design$ is the matrix defined in \eqref{eq:design},
then we have the linear model representation:
\begin{equation}
\label{eq:def-convexs-design}
    \convexs =
    \left\{
        \design^2 \vtheta,
        \qquad
        \vtheta=(\theta_1,...,\theta_n)^T\in\Rn:
        \quad
        \theta_k \ge 0 \text{ for all } k\ge 3
    \right\}.
\end{equation}
If $\vu\in\convexs$ and $\vu = \design^2 \vtheta$ with $\vtheta \ge \vzero$ as in \eqref{eq:def-convexs-design},
then $q(\vu) - 1  \le |\{i=3,...,n: \theta_i > 0\}|$.

\section{Discussion on SOI and OI}

If $K$ is a closed convex set,
there are general methods to bound \eqref{eq:loss} from above
if $\vmu \in K$ and $\hmu=\ls(K)$.
For example, the analysis of \cite{chatterjee2014new}
ensures that $\eqref{eq:loss}$ is bounded from above
by $t_{\sigma,n}(K)^2/n$ with high probability if
\begin{equation}
    \label{eq:cond-gaussian-width}
    \Evmu \sup_{\vu\in K:\, \euclidnorm{\vu-\vmu}\le t_{\sigma,n}(K)}
    \vxi^T(\vu - \vmu)
    \le \frac{t_{\sigma,n}(K)^2}{2},
\end{equation}
where $t_{\sigma,n}(K)$ is a constant that may depend on $\sigma,n$ and $K$.
The left hand side of \eqref{eq:cond-gaussian-width}
is sometimes called the localized Gaussian width of $K$.
However, if $\vmu\notin K$,
to our knowledge there is no general method
to control the regret \eqref{eq:regret} with $E=K$
using conditions similar to \eqref{eq:cond-gaussian-width}
on the localized Gaussian width.
One of the goals of the 
present paper is to fill this gap.
\Cref{thm:isomorphic} allows us to derive inequalities of the form
\begin{equation}
    \label{eq:example-res-soi}
    \scalednorms{\ls(K)-\vmu}
    \le
    \min_{\vu\in K}
    \scalednorms{\vu-\vmu}
    + \frac{c t_{\sigma,n}(K)^2}{n},
\end{equation}
in expectation and with high probability for some absolute constant $c>0$, as soon as
\begin{equation}
    \Evmu \sup_{\vu\in K:\, \euclidnorm{\vu-\Pi_K(\vmu)}\le t_{\sigma,n}(K)}
    \vxi^T(\vu - \Pi_K(\vmu))
    \le \frac{t_{\sigma,n}(K)^2}{2}
\end{equation}
is satisfied, where $\Pi_K(\vmu)$ is the projection of $\vmu$ onto $K$.
The inequality \eqref{eq:example-res-soi}
is a sharp oracle inequality,
i.e., an oracle inequality with leading constant 1,
which yields an upper bound on the regret $R_2$ in
\eqref{eq:regret} with $E$
replaced by $K$.
This is opposed to other oracle inequalities of the form
\begin{equation}
    \label{eq:example-oi}
    \Evmu
    \scalednorms{\hmu-\vmu}
    \le
    C \min_{\vu\in K}
        \scalednorms{\vu-\vmu}
        +
        \frac{t_{\sigma,n}(K)^2}{n},
\end{equation}
where the leading constant $C$ is strictly greater than 1.
The right hand sides of \eqref{eq:example-res-soi} and \eqref{eq:example-oi} are minimized
if $\vu$ is the projection of $\vmu$ onto $K$,
i.e.,
$\min_{\vu\in K}
\scalednorms{\vu-\vmu} = \scalednorms{\Pi_K(\vmu)-\vmu}$
where $\Pi_K:\Rn\rightarrow \K$ is the projection onto $\K$.

A consequence of \Cref{prop:misspecification}
is that an upper bound on the regret of order 2 
is stronger than an upper bound on the estimation error
\eqref{eq:estimation-error}.
Thus, sharp oracle inequalities such as 
\eqref{eq:soi-tstar} always
imply upper bounds on
$\scalednorms{\ls(K) - \Pi_K(\vmu)}$.
On the other hand, oracle inequalities such as \eqref{eq:example-oi}
with leading constant $C$ strictly greater than 1 do not imply
upper bounds on the estimation error
\eqref{eq:estimation-error}.

\section{Unimodal approximation}

For any $V\ge 0$, define the set $E_V = \left\{ m\in \mathbb N: m \ge \left(\frac{V^2 n}{\sigma^2\log(en)}\right)^{1/3}\right\}$
and the integer $k^*(V)$ by $k^*(V) = 1$ if $E_V$ is empty
and $k^*(V) = \min E_v$ otherwise.
\begin{lemma}[Lemma 2 in \cite{bellec2015sharp}]
    \label{lemma:approximation-increasing}
    Let $\vmu\in\increasings$ and let $
    V \ge \mu_n - \mu_1$.
    For any $k=1,...,n$,
    there exists a sequence $\bar\vu\in\increasings$ with $k(\vu) \le k$
    such that
    $\inftynorm{\bar\vu - \vmu} \le \frac{V}{2k}$. 
    Next, there exists a sequence $\vu^*\in\mathcal U$ with $k(\vu^*) \le k^*(V)$
    such that
    \begin{equation}
        \inftynorm{\vu^* - \vmu}
        \le
        \frac{1}{4}
        \max\left( 
            \left(\frac{\sigma^2 V \log(en)}{n}\right)^{2/3}
            ,
            \frac{\sigma^2 \log(en)}{n}
        \right)
        .
        \label{eq:approx-vustar}
    \end{equation}
    In addition, 
    \begin{equation}
        \frac{\sigma^2 k^*(V)}{n} \log\frac{en}{k^*(V)}
        \le
        2 \max\left( 
            \left(\frac{\sigma^2 V \log(en)}{n}\right)^{2/3}
            ,
            \frac{\sigma^2 \log(en)}{n}
        \right)
        .
        \label{eq:algebra-log-k}
    \end{equation}
\end{lemma}

\begin{lemma}
    \label{lemma:approximation}
    Let $\vmu\in\mathcal U$ and let $
    V \ge \max_{i=1,...,n} \mu_i - \min_{i=1,...,n}\mu_i$.
    For any $k=1,...,n$,
    there exists a sequence $\bar\vu\in\mathcal U$ with $k(\vu) \le 2k$
    such that
    $\inftynorm{\bar\vu - \vmu} \le \frac{V}{2k}$. 
    Next, there exists a sequence $\vu^*\in\mathcal U$ with $k(\vu^*) \le 2 k^*(V)$
    such that \eqref{eq:approx-vustar} holds.
\end{lemma}
\begin{proof}
    To deduce \Cref{lemma:approximation} from
    \Cref{lemma:approximation-increasing}, one can proceed as follows.
    If $\vmu\in\mathcal U$ is non-increasing on $\{1,...,m\}$ and 
    nondecreasing on $\{m+1,...,n\}$, we apply \Cref{lemma:approximation}
    to the sequences $-\vmu_{\{1,...,m\}}$ and $\vmu_{\{m+1,...,n\}}$. 
    This yields the existence of $\vu^*_1\in\mathcal S^\uparrow_m$
    and $\vu^*_2\in\mathcal S^\uparrow_{n-m}$ such that $k(\vu^*_1)\le k^*(V)$
    and $k(\vu^*_2)\le k^*(V)$.
    The concatenation $\vu^* \coloneqq (-\vu^*_1,\vu^*_2)$ is unimodal,
    satisfies $k(\vu^*) \le k(\vu^*_1) + k(\vu^*_2) \le 2k^*(V)$
    and is such that \eqref{eq:approx-vustar} holds.
    By the same argument, for any $k=1,...,n$ there exists $\bar \vu \in\mathcal U$
    with $k(\bar\vu) \le 2k$
    such that $\inftynorm{\bar\vu - \vmu} \le \frac{V}{2k}$.
\end{proof}
By simple algebra, we deduce from \eqref{eq:algebra-log-k} that
\begin{equation}
    \frac{\sigma^2 (2k^*(V) + 1)}{n} \log\frac{en}{2 k^*(V) + 1}
    \le
    \max\left( 
        4 \left(\frac{\sigma^2 V \log(en)}{n}\right)^{2/3}
        ,
        \frac{3\sigma^2 \log(en)}{n}
    \right)
    .
    \label{eq:algebra-log-2k+1}
\end{equation}